%% file: root_arxiv.tex
\documentclass{article}
\usepackage[utf8]{inputenc}
\usepackage[letterpaper, margin=1in]{geometry}
\title{Analysis and Control of Input-Affine Dynamical Systems \\ using Infinite-Dimensional Robust Counterparts}
\renewcommand\footnotemark{}

\author{Jared Miller$^{1, 2}$, Mario Sznaier$^2$
\thanks{$^1$J. Miller is with the Automatic Control Laboratory (IfA), Department of Information Technology and Electrical Engineering (D-ITET), ETH Z\"{u}rich, Phisykstrasse 3, 8092, Z\"{u}rich, Switzerland (e-mail: jarmiller@control.ee.ethz.ch).}
\thanks{$^2$J. Miller and M. Sznaier are with the Robust Systems Lab,  ECE Department, Northeastern University, Boston, MA 02115. (e-mails: msznaier@coe.neu.edu).}
\thanks{ J. Miller and M. Sznaier were partially supported by National Science Foundation (NSF) grants  CNS--1646121, CMMI--1638234, ECCS--1808381 and CNS--2038493,  and Air Force Office of Scientific Research (AFOSR) grant FA9550-19-1-0005.  J. Miller was partially supported by the Swiss National Science Foundation Grant 200021\_178890 (NCCR: Automation), and by the Chateaubriand Fellowship of the Office for Science \& Technology of the Embassy of France in the United States.
}
}

\usepackage{graphicx}
\usepackage{subcaption}
\usepackage{amsmath}
\usepackage{amssymb}
\usepackage{mathtools}
\usepackage{hyperref}

\input{macros/jared_macro_arxiv}

\begin{document}

\maketitle
\input{sections/abstract}
\input{sections/introduction}
\input{sections/contributions}
\input{sections_arxiv/summary_arxiv}

\input{sections/preliminaries}

\input{sections/analysis_control_problems}
\input{sections/lie_robustified}
\input{sections/poly_approx}
\input{sections/discrete_time}
\input{sections/data_driven}


 \input{sections_arxiv/examples}
 \input{sections/discrete_time_example}
 \input{sections_arxiv/reach_example}

 \input{sections_arxiv/roa_example}
\input{sections/conclusion}

\section*{Acknowledgements}
The authors would like to thank Didier Henrion, Milan Korda, Victor Magron, and Roy Smith for their advice and support. The authors would also like to thank Carsten Scherer for discussions about the proof of Proposition \ref{eq:polytope_slater}.

\bibliographystyle{IEEEtran}
\bibliography{references}

\appendix
     \input{appendix/appendix_polynomial_aux}

\input{appendix/appendix_continuous.tex}

\input{appendix/appendix_polynomial_mult_AG}

\input{appendix/appendix_robust_occ_duality}

\input{appendix/app_problems}
\input{appendix/appendix_control_cost}
\end{document}

%% file: macros/jared_macro_arxiv.tex
\usepackage[printonlyused]{acronym}
\usepackage{xcolor}
\usepackage{amsmath}
\usepackage{amssymb}
\usepackage{mathtools}
\usepackage{subcaption}
\usepackage{hyperref}
\hypersetup{colorlinks=true, linkcolor=black}
\usepackage[ruled]{algorithm2e}
\usepackage{mathrsfs}

\definecolor{dkgreen}{rgb}{0,0.4,0}
\definecolor{dkred}{rgb}{0.7,0,0}

\newcommand{\R}{\mathbb{R}} 
\newcommand{\K}{\mathbb{K}} 
 
\newcommand{\N}{\mathbb{N}}

\newcommand{\Dc}{\mathcal{D}}
\newcommand{\Lie}{\mathcal{L}} 
 
\newcommand{\Z}{\mathbb{Z}}

\newcommand{\M}{\mathbb{M}}

\newcommand{\0}{\mathbf{0}}
\newcommand{\1}{\mathbf{1}}

\newcommand{\psd}{\mathbb{S}}
\newcommand{\scL}{\mathscr{L}}

\newcommand{\bbm}{\mathbf{m}}

\DeclarePairedDelimiter{\abs}{\lvert}{\rvert}
\DeclarePairedDelimiter{\norm}{\lVert}{\rVert}

\DeclarePairedDelimiter{\floor}{\lfloor}{\rfloor}
\DeclarePairedDelimiter{\Tr}{\textrm{Tr}(}{)}
\DeclarePairedDelimiter{\rank}{\textrm{rank}(}{)}
\DeclarePairedDelimiter{\diag}{\textrm{diag}(}{)}

\DeclarePairedDelimiterX{\inp}[2]{\langle}{\rangle}{#1, #2}
\DeclarePairedDelimiter{\supp}{\textrm{supp}(}{)}
\DeclarePairedDelimiter{\Mp}{\mathcal{M}_+(}{)}

\usepackage{amsthm}
\newtheorem{thm}{Theorem}[section]
\newtheorem{lem}[thm]{Lemma}
\newtheorem{prop}[thm]{Proposition}
\newtheorem{proposition}[thm]{Proposition}

\newtheorem{cor}{Corollary}
\newtheorem{defn}{Definition}[section]

\newtheorem{rmk}{Remark} 

\newcommand{\bpf}{\begin{proof}}
\newcommand{\epf}{\end{proof}}

%% file: sections/abstract.tex
\begin{abstract}
Input-affine dynamical systems often arise in control and modeling scenarios, such as the data-driven case when state-derivative observations are recorded under bounded noise.
Common tasks in system analysis and control include optimal control, peak estimation, reachable set estimation, and maximum control invariant set estimation. Existing work poses these types of problems as infinite-dimensional linear programs in auxiliary functions with sum-of-squares tightenings. The bottleneck in most of these programs is the Lie derivative nonnegativity constraint posed over the time-state-control set. 
Decomposition techniques to improve tractability by eliminating the control variables include  vertex decompositions (switching), or  facial decompositions in the case where the polytopic set is a scaled box. This work extends the box-facial decomposition technique to allow for a robust-counterpart decomposition of semidefinite representable sets (e.g. polytopes, ellipsoids, and projections of spectahedra). These robust counterparts are proven to be equivalent to the original Lie constraint under mild compactness and regularity constraints.
Efficacy is demonstrated under peak/distance/reachable set data-driven analysis problems and Region of Attraction maximizing control. 
\end{abstract}

%% file: sections/introduction.tex
\section{Introduction}

This paper will focus on analysis of an input-affine and continuous-time dynamical system under an input process $w(t)$ with
\begin{equation}
\label{eq:disturbance_affine}
    \dot{x}(t) = f(t, x(t), w(t)) = f_0(t,x) + \textstyle\sum_{\ell=1}^L w_\ell(t) f_\ell(t,x),
\end{equation}
in which the state $x \in X \subset \R^n$ and the input  $w \in W \subset \R^L$ are assumed to lie in compact sets.
The time horizon $t \in [0, T]$ is finite for convergence purposes. It is further required that the set $W$ is an $L$-dimensional compact \ac{SDR} set (a projection of a spectahedron) with non-empty interior \cite{helton2010semidefinite}. 

\Iac{SDR} set could arise from a sequence of observations of $\dot{x}(t)$ as corrupted by bounded noise. An example of such \iac{SDR} set $W$ is the $L$-dimensional polytope described by $m$ constraints (up to $m$ faces), which may be expressed as
\begin{align}
\label{eq:w_set}
    W &= \{w \mid A w \leq b\} & A \in \R^{m \times L}, \ b \in \R^m.
\end{align}

Letting $x_0 \in X_0$ be an initial condition and $w(t)$ be a control with $w(t) \in W \ \forall t \in [0, T]$ and $T$ finite, the state obtained by following dynamics in \eqref{eq:disturbance_affine} is
\begin{equation}
    x(t) = x(t \mid x_0, w(\cdot)).
\end{equation}

The process $w(t)$ is well-defined for each value of $t \in [0, T]$, but is not required to be continuous.

The problem instances that will be addressed in this paper are peak estimation, distance estimation, reachable set estimation, and \ac{ROA} maximization. Each problem instance may be cast as an infinite-dimensional \ac{LP} and approximated through the moment-\ac{SOS} hierarchy. They each have a Lie derivative nonnegativity constraint that usually induces the largest PSD matrix by numerical solvers. Such a constraint may be split using infinite-dimensional robust counterparts \cite{ben2009robust} into smaller \ac{PSD} matrix constraints using convex duality \cite{boyd2004convex} and a theorem of alternatives  \cite{ben2015deriving}.
Decomposition of \ac{SDR} sets $W$ move beyond the previously considered box cases in \cite{majumdar2014convex} \cite{korda2015controller} and polytope case in \cite{miller2021facial_short}.

Peak estimation finds an initial condition $x_0$ and input $w$ that maximizes the instantaneous value of a state function $p(x(t))$ along a trajectory \cite{cho2002linear}. Distance estimation is a variation of peak estimation that finds the distance of closest approach between points along trajectories $x(t \mid x_0, w)$ and an unsafe set $X_u$ \cite{miller2021distance}.
Reachable set estimation identifies the set of points $X_T$ such that there exists a pair $x_0, w$ where $x(T \mid x_0, w) \in X_T$ \cite{aubin2011viability}. Peak and reachable set estimation under input-affine and \ac{SDR} constraints may arise from the data-driven setting where state-derivative observations $\mathcal{D} = \{(t_k, x_k, {y}_k)\}_{k=1}^{N_s}$ are available subject to an $L_\infty,$ $L_2$, or semidefinite-bounded noise process $\eta$ ($y(t_k) \doteq \dot{x}(t_k)+\eta_k$).  As an example,  $L_\infty$-bounded noise could arise from propagating finite-difference errors from  when estimate $\dot{x}$.


Basin/\ac{ROA} estimation problems based on \ac{BRS} approaches have similar principles as forward Reachable set estimation problems. \cite{mitchell2003overapproximating}. \ac{ROA} maximization chooses a control scheme $w(t)$ that optimizes the volume of the set $X_0$ such that initial conditions $x_0 \in X_0$ will land in a goal set $X_T$.

Application of infinite-dimensional \acp{LP} to problems in control theory began with optimal control in \cite{lewis1980relaxation}. Solution methods for these infinite-dimensional \acp{LP} include gridded discretization and the moment-\ac{SOS} hierarchy of \acp{SDP} in \cite{henrion2008nonlinear}. Peak estimation \acp{LP} were solved by discretization in \cite{cho2002linear} and through \ac{SOS} methods in \cite{fantuzzi2020bounding}, with moment-based optima recovery in \cite{miller2020recovery}. Peak estimation under uncertainty was treated in \cite{miller2021uncertain}, and this work will use the input-affine and \ac{SDR} structure $W$ to generate simpler \acp{SDP}.

Reachable set estimation using \acp{LP} occurs from outside in \cite{Henrion_2014} and from inside in \cite{korda2013inner}.
\acp{SDP} associated with the moment-\ac{SOS} hierarchy will produce polynomial sublevel sets that converge in volume to the true reachable set as the polynomial degree increases (under mild conditions, and outside a set of measures zero). Controllers may be formulated to maximize the \ac{BRS}, in which the volume of the set of initial conditions $X_0$ that can be steered towards a target set $X_T$ is maximized \cite{majumdar2013control}. Other approaches towards reachable set estimation of nonlinear systems includes ellipsoidal methods \cite{kurzhanski2002ellipsoidal}, polytopes \cite{harwood2016efficient},  and interval methods using mixed monotonicity \cite{coogan2020mixed}.
Infinite-dimensional \acp{LP} have also been applied to region of attraction estimation and backwards reachable set maximizing control \cite{Henrion_2014}. 
We note that recent work in \cite{bramburger2023auxiliary} involves data-driven estimation of the Lie derivative/infinitesimal generator for use in auxiliary function programs (e.g. time-averages).


The power of infinite-dimensional robust counterparts in reducing the dimensionality of \ac{PSD} matrices can be demonstrated on a peak estimation task in dimension $n=2$, possessing disturbance-affine cubic polynomial dynamics (discussed further in Section \ref{sec:flow_poly}). The uncertainty process $w(t)$ is restricted to an polytope $W$ with $L=10$ dimensions, $N_f = 33$ faces, and $N_v = 7534$ vertices.
The \ac{PSD} constraint of maximal size involved in a degree $d=4$ \ac{SOS} tightening of a peak estimation problem shrinks from 8568 (prior work in \cite{miller2021uncertain}) down to 56 (current work) after application of the infinite-dimensional robust counterpart to eliminate the uncertainty variables. Performing a size-8568 dense \ac{PSD} constraint in solvers such as Mosek or Sedumi is intractable. Applying a vertex decomposition would require that $N_v = 7534$ \ac{PSD} constraints of size 56 hold. In contrast, the robust counterpart in introduced in this work needs only $N_f +1= 34$ \ac{PSD} constraints of size 56.

%% file: sections/contributions.tex
The contributions of this paper are,
\begin{itemize}
    \item Formulation of infinite-dimensional robust counterparts with respect to \ac{SDR} sets along with conditions for their nonconservatism
    \item Evaluation of robust counterparts for Lie constraints in input-affine nonlinear systems analysis with \ac{SDR} input sets $W$
    \item Quantification of \ac{SDP} computational complexity reduction as compared to original programs
    \item Demonstration of robust counterparts on peak and reachable set problems
\end{itemize}

Portions of this paper were presented at the 10th IFAC Symposium on Robust Control Design (ROCOND 2022) \cite{miller2021facial_short}. Additional content present in this paper beyond the conference version includes:
\begin{itemize}
    \item Generalization of facial (polytopes) decompositions to robust counterparts of \ac{SDR} sets
    \item Distance and (Backwards) Reachable set estimation problems 
    \item Robustification of discrete-time dynamics
    \item Additional examples
    \item Proofs of lower semicontinuity, polynomial approximability, and nonconservatism    
    \item Duality with measure programs and controller recovery
\end{itemize}

%% file: sections_arxiv/summary_arxiv.tex

This paper has the following layout: 
Section \ref{sec:preliminaries} provides an overview of preliminaries including notation, the analysis and control problems, robust counterparts, and \ac{SOS} theory.
Section \ref{sec:lie_decomposed} splits the Lie derivative constraint over the \ac{SDR} uncertainty $W$ through the use of an parameterized robust counterpart. 
Section \ref{sec:robust_poly_approx} details \ac{SOS} programs for the robust counterparts. Section \ref{sec:discrete_time} applies parameterized robust counterparts to simplify problems with discrete-time dynamics.
Section \ref{sec:data_driven_background} reviews background of the polyhedral structure of consistency constraints induced by model structures and $L_\infty$-bounded noise processes. Section \ref{sec:robust_examples} presents examples of robust counterparts acting on all four problems. Section  \ref{sec:poly_conclusion} concludes the paper.

%% file: sections/preliminaries.tex
\section{Preliminaries}
\label{sec:preliminaries}

\subsection{Acronyms/Initialisms}
\input{sections/acronym}
\subsection{Notation}


The set of real numbers and its $n$-dimensional Euclidean space are $\R$ and $\R^n$. The notation $\mathbf{0}$ and $\mathbf{1}$ will denote vectors of all zeros and all ones respectively. The set of natural numbers is $\N$, and the subset of natural numbers between $1$ and $n$ is $1..n$.
An element-wise partial ordering $x \geq y$ exists between $x, y \in \R^n$ if $x_i \geq y_i$ for each coordinate $i = 1..n$. All vectors in the nonnegative orthant $x \in \R^n_{\geq 0} \subset \R^n$ satisfy $x \geq \mathbf{0}$. The positive orthant is $\R^n_{>0}$. The inner product between two vectors $x, y \in \R^n$ may be written as $x \cdot y = x^T y = \sum_i x_i y_i$.

The set of $m \times n $ matrices with real values is $\R^{m \times n}$. The set of symmetric matrices $\psd^n$ satisfy $Q = Q^T, Q \in \R^{n\times n}$. The set of \ac{PSD} matrices is $\psd_+^n$ $(Q \succeq 0)$, and the set of \ac{PD} matrices is $\psd_{++}^n$ $(Q \succ 0)$.
The duality paring between two symmetric matrices $A, B \in \psd^n$ is $\inp{A}{B}_{\psd^n} = \sum_{ij} A_{ij} B_{ij}.$


A multi-index is a member of $\N^n$ for finite $n$. The degree of a multi-index $\alpha \in \N^n$ is $\deg \alpha = \sum_i \alpha_i$. A multi-index $\alpha$ is finite if its degree is finite.
A monomial $x^\alpha = \prod_{i=1}^n x^{\alpha_i}$ is a function in an indeterminate value $x$ for a finite multi-index $\alpha$. 
A polynomial with real coefficients $p(x) \in \R[x]$ may be expressed as the sum $p(x) = \sum_{\alpha \in \mathcal{A}} c_\alpha x^\alpha$ over a finite-cardinality set of finite multi-indices $\mathcal{A}$ with bounded coefficients $\{c_\alpha\}_{\alpha \in \mathcal{A}}$. The degree of a polynomial is $\deg p = \max_{\alpha \in \mathcal{A}} \deg \alpha$. The set $\R[x]_{\leq d}$ is the set of polynomials with degree at most $d$. An $n$-dimensional vector of polynomials is $\R[x]^n$.

\subsection{Analysis}
A cone $K$ is a set such that $\forall c > 0$, $x \in K \implies c x \in K$ \cite{boyd2004convex}. A cone $K$ defines a partial ordering $\geq_K$ as $x_1 \geq_K x_2$ if $x_1-x_2 \in K$. The cone $K$ is pointed if $x \in K$ and $-x \in K$ implies that $x=0$. If the cone $K$ is finite-dimensional ($K \subseteq \R^n$), the dual $K^*$ of the cone $K$ is the set $\{y \in \R^n \mid x^T y \geq 0 \ \forall x \in K\}$. 

Let $S \subset \R^n$ be a space, and let $\varnothing$ be the empty set. The ring of continuous functions over $S$ is $C(S)$, and its subring of functions with continuous first derivatives is $C^1(S) \subset C(S)$. The infinite-dimensional cone of nonnegative functions over $S$ is $C_+(S)$.
The $C^0$ norm of a function $f \in C^0(S)$ is $\norm{f}_{C^0(S)} = \sup_{s \in S} \abs{f(s)}$. The $C^1$ norm of $f \in C^1(S)$ is $\norm{f}_{C^1(S)} = \norm{f}_{C^0(S)} + \sum_{i=1} \norm{\partial_i f}_{C^0(S)}$.

\subsection{Robust Counterparts}
\label{sec:prelim_robust_counterpart}


\begin{defn}[\cite{helton2010semidefinite}]
Let $S \subset \R^m$ be a set and let $K \subseteq \R^n$ be a cone. The set $S$ is $K$-representable if there exists a finite dimension $q$  and matrices $A \in \R^{n \times m}, G \in \R^{n \times q}, \ e \in \R^{n}$ such that
\begin{equation}
    S = \{x \in \R^m \mid \exists \lambda \in \R^q: A x + G \lambda + e \in K\}.
\end{equation}

The set $S$ is \ac{SDR} if $K$ is a subset of the \ac{PSD} cone (where \ac{PSD} cone may be vectorized as in $[a, b; \ b, c] \in \psd^2_+ \rightarrow (a,b,c) \in K$).
\end{defn}

\ac{SDR} sets are also referred to as `projections of spectahedra' or `spectahedral shadows', and are closed under the projection, product, and intersection operations. \ac{SDR} sets form a strict subset of all convex sets. 
This paper will focus on three specific self-dual cones $K$ to define \ac{SDR} sets:
\begin{enumerate}
    \item Nonnegative ($\R_{\geq 0}$)
    \item \ac{SOC}/Lorentz ($Q^n: \{(u, v) \in \R^n \times \R_{\geq 0} \mid \norm{u}_2 \leq v\}$)
    \item Positive Semidefinite ($\psd_+^n$)
\end{enumerate}

 Define the constraint vectors $a_0, a_{\ell} \in \R^r$ and $b_0, b_\ell \in \R$  for all $\ell=1..L$. Define $W$ as the intersection of $N_s$ \ac{SDR} sets with cones $K_1..K_s$ as
\begin{align}
    W &= \{w \in \R^L: \forall s = 1..{N_s}, \qquad \exists \lambda_s \in \R^{q_s}:  \quad A_s w + G_s \lambda_s + e_s \in K_s\}.\label{eq:w_sdr}
    \intertext{The following systems each have a robust semi-infinite linear inequality constraint in $\beta \in \R^r$ that must hold for all uncertain values $w$ in an \ac{SDR}:}
\label{eq:lin_robust}
   \textstyle \textrm{Non-strict}: \qquad &\textstyle \forall w \in W: \qquad a_0^T \beta + \sum_{\ell=1}^L w_\ell a_\ell^T \beta \leq b_0 + \sum_{\ell=1}^L w_\ell b_\ell \\
    \textrm{Strict}: \qquad & \textstyle \forall w \in W: \qquad a_0^T \beta + \sum_{\ell=1}^Lw_\ell a_\ell^T \beta < b_0 + \sum_{\ell=1}^L w_\ell b_\ell. \label{eq:lin_robust_strict}
\end{align}

\begin{defn}[Equation (1.3.14) of \cite{ben2009robust} ]
The \textbf{robust counterpart} of \eqref{eq:lin_robust} with respect $w \in W$ is the conic  set of constraints in variables $\{\zeta_s\}_{s=1}^{N_s}$
    \begin{subequations}
    \label{eq:robust_counterpart}
    \begin{align}
    &\textstyle \sum_{s=1}^{N_s} e_s^T \zeta_s + a_0^T \beta \leq b_0  \label{eq:robust_counterpart_crit} \\
    & G_s^T \zeta_s = 0 & & \forall s = 1..N_s \\
    & \textstyle \sum_{s=1}^{N_s} (A_s^T \zeta_s)_\ell + a_\ell^T \beta = b_\ell & & \forall \ell=1..L\\
        & \zeta_s \in K_s^* & & \forall s = 1..N_s.
    \end{align}
    \end{subequations}

\end{defn}

\begin{thm}[Theorem 1.3.4 of \cite{ben2009robust}]
\label{thm:sdr_tight}
Assume  that each $K_s$ is a convex and pointed cone with nonempty interior. Further assume  that there exists a Slater point ($\exists \bar{w} \in \R^L, \forall s: \exists \bar{\lambda}_s \in \R^{q_s}\mid A_s \bar{w} + G_s \bar{\lambda}_s + e_s \in \text{int}(K_s)$) if  $K$ is non-polyhedral. Then the semi-infinite program \eqref{eq:lin_robust} is feasible iff the finite-dimensional robust counterpart  \eqref{eq:robust_counterpart} is feasible. Additionally, \eqref{eq:lin_robust} is infeasible iff  \eqref{eq:robust_counterpart} is infeasible.
\end{thm}

\begin{rmk}
    \label{rmk:slater_violation} The proof of Theorem 1.3.4 of \cite{ben2009robust} relies on strong duality (via the non-polyhedral Slater condition for $K$) in order to prove feasibility equivalence. If $K$ is non-polyhedral and there does not exist a Slater point, then by weak alternatives, feasibility of the robust counterpart in \eqref{eq:robust_counterpart} is sufficient but not necessary to prove feasibility of \eqref{eq:w_sdr}.
\end{rmk}

\begin{rmk}
\label{rmk:slater_equality}
    The uncertainty description in \eqref{eq:w_sdr} can be enriched by affine equality constraints $B w = \theta$. The non-polyhedral Slater condition would then require that the strictly feasible $K$-Slater point also satisfies the affine constraints $B w = \theta$ to ensure strong duality \cite{boyd2004convex}.    
    \end{rmk}

\begin{lem}
\label{lem:robust_strict}
Feasibility equivalence of the robust counterpart also holds in the strict case   \eqref{eq:lin_robust_strict} by applying a $<$ comparator to \eqref{eq:robust_counterpart_crit} \cite{dai2020semi}.
\end{lem}

\subsection{Polynomial Matrix Inequalities}

\label{sec:prelim_robust_pmi }

The symbol $\psd^s[x]$ will refer to the set of $n\times n$ symmetric-matrix-valued polynomials in an indeterminate $x$.

The matrix $P \in \psd^n[x]$ is \iac{SOS}-matrix  if there exists a size $s \in \N$ and a polynomial matrix $B \in (\R[x]^{s \times n}$ such that $P = B^T B$. \ac{SOS} matrices can be characterized by the existence of a Gram matrix $Q \in \psd_+^{s n},$ and a polynomial vector $v(x) \in \R[x]^s$,  such that 
(Lemma 1 of \cite{scherer2006matrix})
\begin{equation}
    P(x) = (v(x) \otimes I_s)^T Q (v(x) \otimes I_s).
\end{equation}

The cone of \ac{SOS} matrices of size $n$ is $\Sigma^n[x] \subset \psd^n[x]$, and its degree-$2d$ truncation is $\Sigma^n_{\leq 2d} \subset \Sigma^n[x]$. The scalar \ac{SOS} cone may be written as $\Sigma^1[x] = \Sigma[x]$.
The \ac{BSA} set $\K$ in this paper will be expressed as the locus of \ac{PSD} polynomial matrix constraints in matrix constraint terms $G_i(x) \in \psd^{n_i}(x)$:
\begin{align}
    \K = \{x \in \R^n \mid G_i(x) \succeq 0 \ \forall i = 1..N_c\}. \label{eq:pmi_set}
\end{align}

Let $q(x) \in \R[x]$ be a polynomial. \Iac{PMI} over the scalar $q$ with respect to the region $\K$ is
\begin{align}
    q(x) &\geq 0 & \forall x \in \K.
\end{align}

The Scherer Psatz proving that $q(x) > 0$ over $\K$ is the statement that (Corollary 1 of \cite{scherer2006matrix})
\begin{subequations}
\label{eq:scherer_psatz}
\begin{align}
    q(x) &= \sigma_0(x) + \textstyle \sum_{i=1}^{N_c} \inp{G_i(x)}{\sigma_i(x)}_{\psd^n} + \epsilon  \label{eq:scherer_psatz_expr}\\
    & \sigma_0 \in \Sigma[x], \ \forall i \in 1..N_c: \  \sigma_i \in \Sigma^{n_i}[x] \label{eq:scherer_multipliers}, \epsilon > 0.
\end{align}
\end{subequations}

The \ac{WSOS} cone $\Sigma[\K]$ is the cone of all polynomials that admit a representation of the form in \eqref{eq:scherer_psatz} (for $\epsilon \geq 0$). The cone $\Sigma[\K]_{\leq 2d}$ is the set of polynomials of degree $\leq 2d$ that are represented by \eqref{eq:scherer_psatz}.
Note that the Scherer Psatz in \eqref{eq:scherer_psatz} is equivalent to the Putinar Psatz \cite{putinar1993compact} when each constraint term $G_i$ has size $n_i = 1, \ \forall i = 1..N_c$. The set $\K$ in \eqref{eq:pmi_set} is Archimedean if there exists an $R > 0$ such that $R - \norm{x}_2^2$ has a Scherer Psatz \eqref{eq:scherer_psatz} expression. Just as in the Putinar Psatz, the Scherer Psatz describes all positive polynomials over $\K$ when $\K$ is Archimedean (Theorem 2 of \cite{scherer2006matrix}). The process of increasing the degree of $\sigma_0, \sigma_i$ until $q$ has a 
representation in \eqref{eq:scherer_psatz} is an instance of the moment-\ac{SOS} hierarchy \cite{lasserre2009moments}.

Define $n_0 = 1, \ d_0=0$ for the multiplier $\sigma_0$, and define $d_i = \lfloor \deg{G_i}/2 \rfloor$ for the constraints $i=1..N_c$.
The degree-$d$ step of the moment-\ac{SOS} hierarchy involving the cone $\Sigma[\K]_{\leq 2d}$ restricts the multipliers $\sigma_i$ to have maximal degree $2(d - d_i)$ for $i=0..N_c$.
The Gram matrices representing the multipliers $\sigma_i$ \eqref{eq:scherer_multipliers} have size $n_i \binom{n+d-d_i}{d-d_i}$. This Gram size should be compared against the scalarization constraint $\forall y \in \R^{n_i}: y_i^T G_i(x) y_i \geq 0 $ involving $n_i + n$ variables, thus resulting in a combinatorially larger Gram matrix of size $\binom{n+n_i+d-d_i}{d-d_i}$.

Refer to \cite{scherer2006matrix, hol2005sum} for generalizations of the presented Scherer Psatz in \eqref{eq:scherer_psatz}, such as cases where $q(x)$ is a polynomial matrix ($q \in \psd^n[x]$ for $n > 1$) over the set $\K$.

%% file: sections/acronym.tex
\begin{acronym}[WSOS]

\acro{BSA}{Basic Semialgebraic}

\acro{BRS}{Backwards Reachable Set}

\acro{LMI}{Linear Matrix Inequality}
\acroplural{LMI}[LMIs]{Linear Matrix Inequalities}
\acroindefinite{LMI}{an}{a}



\acro{LP}{Linear Program}
\acroindefinite{LP}{an}{a}

\acro{OCP}{Optimal Control Problem}
\acroindefinite{OCP}{an}{a}



\acro{PD}{Positive Definite}

\acro{PMI}{Polynomial Matrix Inequality}

\acro{PSD}{Positive Semidefinite}

\acro{ROA}{Region of Attraction}

\acro{SOC}{Second Order Cone}
\acroindefinite{SOC}{an}{a}

\acro{SDP}{Semidefinite Program}
\acroindefinite{SDP}{an}{a}

\acro{SDR}{Semidefinite Representable}
\acroindefinite{SDR}{an}{a}

\acro{SIR}{Susceptible, Infected, Removed}
\acroindefinite{SIR}{an}{a}

\acro{SOS}{Sum of Squares}
\acroindefinite{SOS}{an}{a}

\acro{TV}{Total Variation}

\acro{WSOS}{Weighted Sum of Squares}
\end{acronym}

%% file: sections/analysis_control_problems.tex
\subsection{Analysis and Control Problems}
\label{sec:analysis_control_problem}
This subsection will present the peak estimation, distance estimation, reachable set estimation, and \ac{ROA} maximization problems along with their auxiliary function-based approximation approaches. The following assumptions will be shared among all problems,
\begin{enumerate}
\item[A1] There is a finite time horizon $T$. \label{assump:time}
\item[A2]
The state sets $X$ and $X_0$ are compact with $X_0 \subset X$. \label{assump:state}
\item[A3] Dynamics $f$ are disturbance-affine \eqref{eq:disturbance_affine}, and all functions $f_0(t, x)$ and $\{f_\ell(t, x)\}_{\ell=1}^L$  are Lipschitz within $[0, T] \times X$. \label{assump:lipschitz}
\item[A4] The input \ac{SDR} set  $W(t, x)$ satisfies the assumptions of Theorem \ref{thm:sdr_tight} (compact, nonempty relative interior, existence of Slater points for non-polyhedral cones $K$) for each $(t, x) \in [0, T] \times X$, and $A(t, x), \ G(t, x), \ e(t, x)$ are all continuous maps.
\item[A5] If $x(t \mid x_0, w(\cdot) \, \in \partial X$ for some $t \in [0, T], \ x_0 \in X_0, w(\cdot) \mid \forall t' \in [0, t]: w(t') \in W$, then $x(t' \mid x_0, \hat{w}(\cdot)) \not\in X \ \forall t' \in (t, T]$ for all $\hat{w}(\cdot)$ such that $\forall s \in [0, t]: \hat{w}(s) = w(s)$ and $\forall s \in [0, T]: \  \hat{w}(s) \in W$. 
\end{enumerate}


\begin{rmk}
    Assumption A5 is a non-return condition in the style of A4 from \cite{miller2021distance}. Once a disturbed trajectory leaves the region $X$, it will never return to $X$ for any applied disturbance.
    A5 can be replaced by a robust invariance property, that $w$-controlled trajectories starting from $X_0$ (with $\forall t \in [0, T]: \ w(t) \in W$) will stay in $X$ for all $[0, T]$. Robust invariance of $X$ implies non-return for any $\bar{X} \supset X$ such that $\partial{\bar{X}} \cap \partial{X} = \varnothing$ (strict superset). Refer to Remark 1 of \cite{miller2021distance} for further discussion of non-return vs. invariance.
\end{rmk}


\subsubsection{Peak Estimation}

The peak estimation problem identifies the supremal value of a state function $p(x)$ attained along trajectories
\begin{align}
    P^* = & \sup_{t^* \in [0, T],\, x_0 \in X_0,\, w(\cdot)} p(x(t^* \mid x_0, w(\cdot))) \label{eq:peak_traj_poly}\\
    & \dot{x}(t) = f(t, x(t), w(t)), \ w(t) \in W \quad  \forall t \in [0, T], \nonumber
    & x(0) = x_0.
    \end{align}

Assumptions on the cost $p(x)$ are added:
\begin{enumerate}
\item[A(Peak)] The cost $p(x)$ is lower semicontinuous inside $X$.
\end{enumerate}

\begin{rmk}
 Further, assumption A(Peak) implies that $p$ is bounded inside $X$, and therefore that $P^*$  is bounded above.
\end{rmk}




\subsubsection{Distance Estimation}
Let $c(x, y)$ be a metric in $X$ and $X_u \subset X$ be a compact unsafe set. The point-unsafe-set distance function is $c(x; X_u) = \inf_{y \in X_u} c(x, y)$. The distance estimation problem from \cite{miller2021distance} will be posed as a peak estimation problem in \eqref{eq:peak_traj_poly} with objective $p(x) = -c(x; X_u)$.
The distance of closest approach $c^*$ obtained by points along trajectories starting at $X_0$ is $c^* = -P^*$.

\subsubsection{Reachable Set Estimation}

The reachability set $X_T$ is the set of all $x$ that can be reached at time index $t=T$ for trajectories starting in the set $X_0$ (under assumptions A1-A4):
\begin{equation}
\label{eq:reach_true_set}
    X_T = \{x(T \mid x_0) \mid x(0) = x_0 \in X_0, \ \exists w(\cdot) \mid \forall t \in [0, T]: w(t) \in W, \ x'(t) = f(t, x, w(t))\}.
\end{equation}

The methods in \cite{henrion2013convex} propose the following volume maximization problem to find the reachable set $X_T$ by
\begin{subequations}
\label{eq:reach_traj}
\begin{align}
P^* = & \sup_{\tilde{X}_T \subset X} \ \textrm{vol}(X_T) \\
&\forall \tilde{x} \in X_T, \exists x_0 \in X_0, \ w(t) \in W : \nonumber \\
& \qquad \tilde{x} = x(T \mid x_0, w(t))\\ 
& x'(t) = f(t, x) \qquad \forall t \in [0, T].
\end{align}
\end{subequations}

The maximal-volume reachable set from \eqref{eq:reach_traj} satisfies $\tilde{X}_T \subseteq X_T$, and is equal to $X_T$ up to a set of measure 0 in volume (e.g., isolated points).

\subsubsection{Region of Attraction Maximization}

Let the compact $X_T \subset X$ be a given `goal' or `target' set. The \ac{BRS}/\ac{ROA} given $X_T$ is the set
\begin{align}
\label{eq:brs}
    X_0 = \{x_0 \mid &x(0) = x_0 \in X_0, \ x'(t) = f(t, x, w(t)), \\ \nonumber
    &x(T \mid x_0, w) \in X_T, w(t) \in W\}. 
\end{align}
Intuitively, the set $X_0$ is the set of states that may be steered towards the goal set $X_T$ in time $T$. The \ac{ROA}-maximization formulation of optimal control aims to find a control scheme that maximizes the volume of $X_0$, similar to how problem \eqref{eq:reach_traj} maximized the volume of $X_T$ to acquire the reachable set. 

%% file: sections/lie_robustified.tex
\section{Decomposed Lie Constraint}
\label{sec:lie_decomposed}

This section provides a framework for decomposing a Lie derivative constraint using robust counterparts, with specific focus on peak estimation.

\subsection{Peak Estimation Program}
The problems in \ref{sec:analysis_control_problem} can be converted into infinite-dimensional \acp{LP} in \textit{auxiliary} functions.
The Lie derivative $\Lie_f v(t, x)$ of a scalar  auxiliary function $v(t, x) \in C^1([0, T] \times X)$ with respect to dynamics $\dot{x}(t) = f(t, x(t), w(t))$
\begin{subequations}
\label{eq:lie}
\begin{align}
    \Lie_f v &= \partial_t v(t,x) + \nabla_x v(t,x) \cdot f(t,x,w) \\
    \intertext{The specific form of the Lie derivative w.r.t. input-affine dynamics  \eqref{eq:disturbance_affine} is}
    \Lie_f v &= \Lie_{f_0} v(t,x) + \textstyle \sum_{\ell=1}^L \nabla_x v(t,x) \cdot w_\ell f_\ell(t,x).
\end{align}
\end{subequations}

An infinite-dimensional \ac{LP} for peak estimation \eqref{eq:peak_traj_poly} with variables $v(t, x)\in C^1([0, T]\times X), \ \gamma \in \R$ under a time-varying disturbance process $w(t) \in W$ is \cite{miller2021uncertain}
\begin{subequations}
\label{eq:poly_peak_w}
\begin{align}
    d^* = & \ \inf_{v, \gamma} \quad \gamma & \\
    & \gamma \geq v(0, x)  & &  \forall x \in X_0 \\
    & \Lie_f v(t, x, w) \leq 0 & & \forall (t, x, w) \in [0, T] \times X \times W \label{eq:poly_peak_w_lie}\\
    & v(t, x) \geq p(x) & & \forall (t, x) \in [0, T] \times X  \label{eq:poly_peak_w_cost} \\
    & v(t, x)\in C^1([0, T]\times X).
\end{align}
\end{subequations}

The auxiliary function $v(t, x)$ is an upper bound on the cost $p(x)$ \eqref{eq:poly_peak_w_cost}, and must decrease along all possible disturbed trajectories \eqref{eq:poly_peak_w_lie}. The $P^* = d^*$ between programs \eqref{eq:peak_traj_poly} and \eqref{eq:poly_peak_w} will match under assumptions A1-A6. The \ac{LP} in \eqref{eq:poly_peak_w} may be approximated through the moment-\ac{SOS} hierarchy, and this sequence of upper bounds (outer approximations) will converge $d_k^* \geq d_{k+1}^* \geq \ldots $ to $P^*$.


\subsection{Parameterized Robust Counterparts}
The Lie derivative \eqref{eq:lie} in constraint \eqref{eq:poly_peak_w_lie}
must respect the constraint
\begin{align}
\label{eq:lie_expand}
    \Lie_f v(t, x, w) & \leq 0 &  \forall (t, x, w) \in [0, T] \times X \times W.
\end{align}

This subsection will lay out a framework of parameterized robust counterparts in order to eliminate the input-affine uncertainty variable $w$ from \eqref{eq:lie_expand}, with respect to a strict inequality 
\begin{align}
\label{eq:lie_expand_strict}
    \Lie_f v(t, x, w) & < 0 &  \forall (t, x, w) \in [0, T] \times X \times W.
\end{align}

Passing from \eqref{eq:lie_expand} to the strict version \eqref{eq:lie_expand_strict} is not burdensome, as $v$ admits a polynomial approximation to arbitrary accuracy in objective satisfying the strict constraint \eqref{eq:lie_expand_strict}:
\begin{thm}
\label{thm:v_approx}
    Given a tolerance $\epsilon>0$, the peak estimation task \eqref{eq:poly_peak_w} (with optimal cost $d^*$) admits a feasible polynomial auxiliary function $V(t, x)$ with objective $d^* + (5/2)\epsilon$ such that $\Lie_f V(t, x) < 0$ holds strictly in $[0, T] \times X$.
\end{thm}
\begin{proof}
    See Appendix \ref{app:poly_robust_aux}.
\end{proof}
\begin{rmk}
Refer to Theorem 3.4 of \cite{miller2021distance} for a similar proof w.r.t. distance estimation, and to \cite{henrion2013convex, majumdar2014convex} for proofs of no relaxation gap (in the sense of volume) for reachable set estimation and for \ac{BRS} maximization.
\end{rmk}

Let $Y$ be a parameter set with parameter value $y \in Y$ (generalizing the choice of $y = (t, x)$ and $Y = [0, T] \times X$ in the case of dynamical systems).  

Define the following quantities based on \eqref{eq:lin_robust}:
\begin{subequations}
\begin{align}    
    Z &= K^* \times \R^r & e &= [e_1; e_2; \ldots; e_{N_s}]\\
    a_\bullet &= [a_1, a_2, \ldots, a_{L}] & b_\bullet &= [b_1; b_2; \ldots; b_{L}]  \\
     A &= \textrm{blkdiag}(A_1, A_2, \ldots A_{N_s}) & G &= \textrm{blkdiag}(G_1, G_2, \ldots G_{N_s}) \\
     q &= \textstyle \sum_{s=1}^{N_s} q_s.
\end{align}
\end{subequations}

The uncertainty set $W$ will be treated as a set-valued map $W: Y \rightrightarrows \R^L$ through parameter-dependent version of the definition in \eqref{eq:w_sdr}:
\begin{align}
        W(y) &= \{w \in \R^L: \lambda \in \R^{q}:  A(y) w + G(y) \lambda + e \in K \}.\label{eq:w_sdr_param}
\end{align}

The map $W$ in \eqref{eq:w_sdr_param} has closed, convex images for each parameter $y \in Y$. The strict robust inequality in \eqref{eq:lin_robust_strict} will be posed over the parameter-dependent set $W$ to find a $\beta(y)$ as
\begin{align}
    \textstyle \forall w \in W(y): \qquad \beta(y)^T(a_0(y) +  a_\bullet(y)  w(y)) < b_0(y) +  b_\bullet(y) w(y). \label{eq:lin_robust_strict_param}
\end{align}

The multiplier set is  $Z = (\prod_{s=1}^{N_s} K^*_s) \times \R^r $ with multipliers $(\zeta, \beta) \in Z$.
The solution map $S: Y \rightrightarrows Z$ of the parameterized robust counterpart \eqref{eq:lin_robust_strict_param} is
\begin{subequations}
\label{eq:sol_zeta_beta}
\begin{align}
    S(y) &= \left\{(\zeta, \beta) \in Z: \begin{array}{r} e^T \zeta + a_0^T\beta  < b_0 \\ G^T \zeta = 0 \\ A^T \zeta + a_\bullet^T \beta = b_\bullet \end{array}\right\}.
\end{align}
\end{subequations}

The following assumptions on \eqref{eq:lin_robust_strict_param} are required:
\begin{itemize}
    \item[A1'] The cone $K$ is convex and pointed.
    \item[A2'] The parameter set $Y$ is compact.
    \item[A3'] The problem data $(a_0, a_\bullet, b_0, b_\bullet,  A, G, e)$ of \eqref{eq:robust_counterpart} are all continuous functions of $y \in Y$.
    \item[A4'] If $K$ is non-polyhedral, then for each $y \in \textrm{dom}(W)$, at least one of the two following options hold:
    \begin{enumerate}
        \item $W(y)$ is a single point  $W(y) = \{w'_y\}$
        \item There exists $(\bar{w}_y, \bar{\lambda}_y) \in W(y) \times \R^{q} $ such that $A(y)\bar{w}_y + G(y) \bar{\lambda}_y + e(y) \in \textrm{int}(K).$ \end{enumerate}
    \item[A5'] For each $y \in Y$, there exists a $\hat{\zeta} \in \textrm{int}(K^*)$ such that $A(y)^T \hat{\zeta} = 0$ and $G(y)^T \hat{\zeta} = 0$.
\end{itemize}

The following assumption is optional:
\begin{itemize}
 \item[A6'] The problem entries $(A, G)$ are constant in $y$.
\end{itemize}

The main result of this section is the following theorem:
\begin{thm}
\label{thm:multipliers_cont_poly}
    Under Assumptions A1'-A5', there exists a selection of the solution map $S(y)$ as in $Y \rightarrow Z: \ y \mapsto (\zeta_c(y), \beta_c(y))$ for the strict parameterized counterpart \eqref{eq:lin_robust_strict_param} such that the functions $\zeta_c(y), \beta_c(y)$ are continuous. Moreover, the functions $\zeta_c(y), \beta_c(y)$ can be taken to be polynomials in $y$ if A6' holds.
\end{thm}
\begin{proof}
    See Appendix \ref{app:continuity_robust} for a proof of continuous selections, and Appendix \ref{app:poly_robust_mult} for polynomial selections.
\end{proof}

\begin{rmk}
    The establishment of polynomial approximability in Appendix \ref{app:poly_robust_mult} is based on the Stone-Weierstrass theorem. Other methods for continuous function approximations can be used instead, such as neural network approximants with increasing width \cite{abate2021fossil}.
\end{rmk}




\subsection{Robust Lie Decomposition}

Constraint \eqref{eq:lie_expand} may be expressed as a semi-infinite linear inequality \eqref{eq:lin_robust} under the correspondence (holding $\forall \ell=1..L$) 
\begin{subequations}
\label{eq:lie_correspondence}
\begin{align}
    b_0 &= -(\partial_t + f_0(t, x)\cdot \nabla_x)v(t, x) = -\Lie_{f_0} v(t, x) & a_0 = 0\\
    b_\ell&= -f_\ell(t, x)\cdot \nabla_x v(t, x) & a_\ell = 0.
\end{align}
\end{subequations}

The parameter set in the Lie setting is $Y = [0, T] \times X$, and the solution set for multipliers is $Z = \prod_{s=1}^{N_s} K^*_s$ with $\beta = \varnothing$.
The robust counterpart of \eqref{eq:lie_expand_strict} with (possibly discontinuous) multiplier variables $\zeta_s(t, x)$ is
    \begin{subequations}
    \label{eq:lie_robust_counterpart}
    \begin{align}
    & \textstyle \Lie_{f_0}v(t, x) + \sum_{s=1}^{N_s} e_s^T \zeta_s(t, x) < 0 & & \forall(t, x) \in  [0, T] \times X \label{eq:lie_robust_counterpart_crit} \\
    & G_s^T \zeta_s(t, x) = 0 & & \forall s = 1..N_s \label{eq:lie_robust_eq_G}\\
    & \textstyle \sum_{s=1}^{N_s} (A_s^T \zeta_s(t, x))_\ell +f_\ell(t, x)\cdot \nabla_x v(t, x) = 0 & & \forall \ell=1..L \label{eq:lie_robust_eq_A}\\
     & \zeta_s(t, x) \in K_s^* & & \forall s = 1..N_s, (t, x) \in [0, T] \times X. \label{eq:lie_robust_zeta}
    \end{align}
    \end{subequations}

\begin{lem}
    Assumptions A1-A5 imply A1'-A4'.
\end{lem}
\begin{proof}
    Under the definition $y = (t, x)$, the compactness assumption A2' is fulfilled by A1 and A2. The conic and Slater structure of A4 complete A1' and A4'. Given that $f_0(t, x)$ and each $f_\ell(t, x)$ are Lipschitz in $[0, T] \times X$, the property that  $v(t, x) \in C^1([0, T] \times X) \implies \Lie v(t, x) \in C([0, T] \times X)$ ensures satisfaction of A3'.
\end{proof}

\begin{cor}
If  $v(t, x) \in C^1([0, T] \times X)$,  then feasibility equivalence of     \eqref{eq:lie_expand_strict} and \eqref{eq:lie_robust_counterpart} holds. Additionally, the function $\zeta(t, x)$ can be chosen to be continuous (polynomial).
\end{cor}
\begin{proof}
    This is a direct application of Theorem \ref{thm:multipliers_cont_poly} with respect to the correspondence in \eqref{eq:lie_correspondence}.
\end{proof}

\begin{rmk}
\label{eq:equality_cont_understood}
    The equality constraints in \eqref{eq:lie_robust_eq_G} and \eqref{eq:lie_robust_eq_G} are understood to hold in the sense of functions $(\forall (t, x) \in [0, T] \times X)$.
\end{rmk}




\begin{rmk}
\label{rmk:equality}
    Consider a parameter set with  equality constraints $W = \{w \in \R^L \in A w + e \in K, B w = \theta \}$ with $B \in \R^{r \times L}, \theta \in \R^r$ (and $G=\varnothing$ for simplicity of explanation) as in Remark \ref{rmk:slater_equality}. 
    The robust counterpart of \eqref{eq:lie_expand} under equality constraints (with multiplier variables $\mu(t, x)$) is
    \begin{subequations}
    \label{eq:lie_robust_counterpart_mu}
    \begin{align}
    & \textstyle \Lie_{f_0}v(t, x) + \sum_{s=1}^{N_s} e_s^T \zeta_s(t, x) + \theta ^T \mu(t, x)\leq 0 & & \forall [0, T] \times X \label{eq:lie_robust_counterpart_crit_mu} \\   
    & \textstyle \sum_{s=1}^{N_s} (A_s^T \zeta_s(t, x))_\ell + B^T_\ell \mu(t, x)+f_\ell(t, x)\cdot \nabla_x v(t, x) = 0 & & \forall \ell=1..L\\
     & \forall s = 1..N_s: \ \zeta_s(t, x) \in K_s^* & & \forall (t, x) \in [0, T] \times X \label{eq:lie_robust_zeta_mu} \\
     & \mu(t, x) \in \R^r & & \forall  (t, x) \in [0, T] \times X. \label{eq:lie_robust_mu_mu}
    \end{align}
    \end{subequations}

The robust counterpart in \eqref{eq:lie_robust_counterpart_mu} can be interpreted in the lens of \eqref{eq:robust_counterpart} to have the correspondence 
\begin{subequations}
\begin{align}
    b_0 &= -\Lie_{f_0} v(t, x), & a_0 &= \theta \\
    b_\ell &= -f_\ell(t, x) \cdot \nabla_x v(t, x), & a_\ell &= B_\ell, 
\end{align}
\end{subequations}
in which the solution $\beta$ to the linear system \eqref{eq:lin_robust} is the value of the multiplier $\mu(t, x)$.

    
\end{rmk}

\subsection{Applications to Dynamical Systems}

We close this section by presenting the robustification of the peak estimation program \ref{eq:poly_peak_w} under a polytope-bounded disturbance $W$:
\begin{subequations}
\label{eq:poly_peak_w_robust}
\begin{align}
    d^* = & \ \inf_{v(t,x), \gamma} \quad \gamma & \\
    & \gamma \geq v(0, x)  & &  \forall x \in X_0 \\
& \Lie_{f_0} v(t, x) + e^T \zeta(t, x) \leq 0 & & \forall (t, x) \in [0, T] \times X  \\
    &  -(A^T)_{\ell}\zeta(t,x) + f_\ell \cdot \nabla_x v(t,x) =0 & & \forall \ell=1..L \\
    & G^T \zeta(t, x) = 0 \\    
    & v(t, x) \geq p(x) & & \forall (t, x) \in [0, T] \times X  \\
    & v(t, x)\in C^1([0, T]\times X) \\
    & \zeta_j(t, x) \in C_+([0, T] \times X) & & \forall j = 1..m.
\end{align}
\end{subequations}

\begin{cor}
\label{cor:peak}
    Under Assumptions A1-A5, programs \eqref{eq:poly_peak_w} and \eqref{eq:poly_peak_w_robust} will have the same value.
\end{cor}
\begin{proof}
    Nonconservatism of the strict of the Lie constraint is verified in \ref{thm:v_approx}. The remaining statements of feasibility equivalence are certified by \ref{thm:multipliers_cont_poly}.
\end{proof}

Appendix \ref{app:analysis_control_problem} lists Linear Programs for the remaining tasks  in Section \ref{sec:analysis_control_problem}, each of which can be robustified using similar arguments to Corollary \ref{cor:peak}.
Appendix \ref{app:integral_cost} formulates robust counterparts to Lie nonnegativity constraints in optimal control under commonly used integral (running) costs.

%% file: sections/poly_approx.tex
\section{Sum-of-Squares Approximation}

\label{sec:robust_poly_approx}
This section develops \ac{SOS} approximations of the infinite-dimensional Lie robust counterpart \eqref{eq:lie_robust_counterpart}.

We now assume polynomial structure on our problem setting:
\begin{itemize}
    \item[A6] The functions $f_0, f_\ell$ are each polynomial, and the set $[0, T] \times X$ is Archimedean.
\end{itemize}

\begin{thm}
\label{thm:robust_zeta_poly}
    Multipliers $\zeta$ in \eqref{eq:lie_robust_zeta} can be chosen to be polynomial when $v$ is polynomial and when \eqref{eq:lie_robust_counterpart_crit} holds strictly.
\end{thm}
\begin{proof}
    See Appendix \ref{app:poly_robust_mult}. When $v$ is polynomial, the vector indexed by $b_\ell = f_\ell(t, x) \cdot \nabla_x v(t, x)$ is also polynomial.
\end{proof}

We now provide details on polynomial approximation and \ac{SOS} implementation over (products of) the nonnegative, \ac{SOC}, and \ac{PSD} cones. The following programs are specific instances of more general semidefinite representations of the \ac{SOS} cones in general algebras, as described in \cite{papp2013semidefinite}.


\subsection{Polytope Restriction}
\label{sec:restriction_polytope}
Assume that the \ac{SDR} set $W$ is the polytope $\{w \mid \exists \lambda \in \R^{L'}:  A w + G \lambda \leq e\}$ for matrices $A \in \R^{m \times L}, G \in \R^{m \times L'}, \ b \in \R^m$. Define $A_s$ as the $s$-th row of $A$ and $(A^T)_\ell$ as the  $\ell$-th column of $A$ (transpose $\ell$-th row of $A^T$).
This case corresponds to $\forall s = 1..m: \ K_s = \R_{\geq 0}$ under the cone description $(-A_s, -G_s, e_s)$ \eqref{eq:w_sdr}.
 The expression of the robustified Lie constraint in \eqref{eq:lie_robust_counterpart} for the polytopic case is
\begin{subequations}
\begin{align}
    & \Lie_{f_0} v(t, x) + e^T \zeta(t, x) \leq 0 & & \forall (t, x) \in [0, T] \times X \label{eq:robust_lie_polytope} \\
    &  -(A^T)_{\ell}\zeta(t,x) + f_\ell \cdot \nabla_x v(t,x) =0 & & \forall \ell=1..L \label{eq:robust_lie_polytope_eq}\\
    & G^T \zeta(t, x) = 0 \\
    & \zeta(t, x) \in \R_{\geq 0}^m.
\end{align}
\end{subequations}

The \ac{SOS} tightening of the constraints in \eqref{eq:robust_lie_polytope} when $(v, \zeta)$ are polynomials is
\begin{subequations}
\label{eq:robust_lie_polytope_sos}
\begin{align}
    &-\Lie_{f_0} v(t, x) - e^T \zeta(t, x) \in \Sigma^1[([0, T] \times X)]\label{eq:robust_lie_polytope_sos_crit} \\
    &  \textrm{coeff}_{t, x}(-A^T \zeta(t,x) + f_\ell \cdot \nabla_x v(t,x)) =0 \label{eq:robust_lie_polytope_sos_eq}\\
    &  \textrm{coeff}_{t, x}(-G^T \zeta(t,x) )=0 \\
    &\zeta_s(t, x) \in \Sigma^1[([0, T] \times X)] & & \forall s = 1..m.\label{eq:robust_lie_polytope_sos_zeta} 
\end{align}
\end{subequations}

The degree-$d$ tightening of program \eqref{eq:robust_lie_polytope_sos} has a Gram matrix of maximal size size $\binom{n+\tilde{d}}{n}$ from \eqref{eq:robust_lie_polytope_sos_crit} and  $m$ Gram matrices of maximal size $\binom{n+d}{d}$ from \eqref{eq:robust_lie_polytope_sos_zeta}.

\subsection{Semidefinite Restriction}

This subsection involves the case where $W$ is an \ac{SDR} set with describing matrices $A_0, A_\ell, G_{k} \in \psd^q$:
\begin{align}
    \textstyle W = \{w \in \R^L \mid A_0 + \sum_{\ell=1}^L w_\ell A_\ell + \sum_{\ell=1}^{L'} \lambda_{k} G_k \succeq 0\}.
\end{align}

The robust counterpart expression in  \eqref{eq:lie_robust_counterpart} is \begin{subequations}
\label{eq:robust_lie_psd}
\begin{align}
    & \Lie_{f_0} v(t, x) + \Tr{A_0 \zeta(t, x)} \leq 0 & & \forall (t, x) \in [0, T] \times X \label{eq:robust_lie_psd_crit} \\
    &  \Tr{A_\ell \zeta(t,x)} + f_\ell \cdot \nabla_x v(t,x) =0& & \forall \ell = 1..L \label{eq:robust_lie_psd_eq}\\
    & \Tr{G_k \zeta(t, x)} = 0 & & \forall k = 1..L' \\
    & \zeta(t, x) \in \psd^q_+.
\end{align}
\end{subequations}

The matrix-\ac{SOS} tightening applied to \eqref{eq:robust_lie_psd} is
\begin{subequations}
\label{eq:robust_lie_psd_sos}
\begin{align}
    & -\Lie_{f_0} v(t, x) - \Tr{A_0 \zeta(t, x)} -\epsilon \in \Sigma^1[([0, T] \times X)] \label{eq:robust_lie_psd_crit_sos} \\
    &  \textrm{coeff}_{t,x}(\Tr{A_\ell \zeta(t,x)} + f_\ell \cdot \nabla_x v(t,x)) =0& & \forall \ell = 1..L \label{eq:robust_lie_psd_eq_sos} \\
    &  \textrm{coeff}_{t,x}(\Tr{G_k \zeta(t, x)} = 0) & & \forall k = 1..L' \\
    & \zeta(t, x)\in \Sigma^q[([0, T]\times X)]. \label{eq:robust_lie_psd_zeta}
\end{align}
\end{subequations}

 The maximal-size Gram matrix at degree $d$ will either occur in \eqref{eq:robust_lie_psd_crit_sos} with size $\binom{n+\tilde{d}}{n}$ or in  \eqref{eq:robust_lie_psd_sos} with size $q\binom{n+d}{d}$.

\subsection{Second-Order Cone Restriction}

This final subsection involves the \ac{SOC} case $W = \{w \in \R^L \mid \exists \lambda \in \R^{L'}: \norm{A w + G \lambda + e}_2 \leq r \}$ for $A \in \R^{m \times L}, \ G \in \R^{m \times L'}, e \in \R^L, \ r \in \R_{\geq 0}$. The constraint in $W$ may be formulated as the \ac{SOC} expression
\begin{equation}
\label{eq:w_soc_set}
    (Aw + G \lambda + e, r) \in Q^m. 
\end{equation}

The robust counterpart \eqref{eq:lie_robust_counterpart} applied to \eqref{eq:w_soc_set} involves a partitioned multiplier function $\zeta = (\beta, \tau) \in Q^m$:
    \begin{subequations}
    \label{eq:robust_counterpart_soc}
    \begin{align}
    & \Lie_{f_0}v(t, x) + r \tau(t, x) - e^T \beta(t, x)\leq 0 \label{eq:robust_counterpart_soc_crit} \\
    & G^T \beta(t, x) = 0 \\
    & (A^T)_\ell \beta(t, x) = f_\ell(t, x) \cdot \nabla_x v(t, x) & & \forall \ell=1..L\\
        & (\beta(t, x), \tau(t, x)) \in Q^m.\label{eq:robust_counterpart_soc_soc}
    \end{align}
    \end{subequations}

\Iac{SOS} formulation of \eqref{eq:robust_counterpart_soc} requires the following lemma:
\begin{lem}
\label{lem:soc_sdp}
The \ac{SOC} membership $(\beta, \tau) \in Q^m$ may be expressed by the following equivalent pairs of \acp{SDP}  \cite{alizadeh2003second}
\begin{subequations}
\begin{align}
       (\beta, \tau) \in Q^m  \Leftrightarrow  & \begin{bmatrix}
           \tau & \beta^T \\ \beta & \tau I
       \end{bmatrix} \in \psd^{m+1}_2 \label{eq:psd_sdp}\\
       & \exists \omega \in \R^m: \ \begin{bmatrix}
        \tau & \beta_j \\ \beta_j & \omega_j 
    \end{bmatrix} \in \psd^2_+, \quad \tau = \textstyle \sum_{j=1}^m \omega_{j} \label{eq:soc_sdp}
\end{align}
\end{subequations}
\end{lem}

\begin{thm}
\label{thm:soc_arch}
    Let $Y \subset \R^m$ be an Archimedean \ac{BSA} parameter set and $(\beta(y), \tau(y))$ be polynomials. Then a necessary and sufficient condition for  $(\beta(y), \tau(y))$ to be in $Q^m$ over the set $Y$ is if there exists polynomials $\omega_j(y) \in \R[y]$ such that:
\begin{equation}
    \begin{bmatrix}
        \tau(y) & \beta_j(y) \\ \beta_j(y) & \omega_j(y) 
    \end{bmatrix} \in \Sigma^2[Y], \quad \tau(y) = \textstyle \sum_{j=1}^m \omega_{j}(y). \label{eq:soc_arch}
\end{equation}
\end{thm}
\begin{proof}
    This relationship holds by Theorem 2.4 of \cite{zheng2023sum} regarding \ac{SOS} correlatively sparse programs, under the requirement that $Y$ is Archimedean.
\end{proof}


Lemma \ref{lem:soc_sdp} and Theorem \ref{thm:soc_arch} will be used to form an \ac{SOS}-matrix representation of \eqref{eq:robust_counterpart_soc_soc}:
    \begin{subequations}
    \label{eq:robust_counterpart_soc_sos}
    \begin{align}
    & -\Lie_{f_0}v(t, x) - r \tau(t, x) + e^T \beta(t, x) \in \Sigma^1[([0, T]\times X)]\label{eq:robust_counterpart_soc_sos_crit} \\
    & \textrm{coeff}_{t, x}(G^T \beta(t, x))= 0 \\
    & \textrm{coeff}_{t, x}((A^T)_\ell \beta(t, x) - f_\ell(t, x) \cdot \nabla_x v(t, x))=0 & & \forall \ell=1..L\\
      &\begin{bmatrix}
        \sum_{j=1}^m\omega_j(t, x) & \beta_j(t, x) \\ \beta_j(t, x) & \omega_j(t, x) 
    \end{bmatrix} \in \Sigma^2[([0, T]\times X)] & & \forall j = 1..m \label{eq:robust_soc_sos_sdp}
    \end{align}
    \end{subequations}

    The degree-$d$ truncation of \eqref{eq:robust_counterpart_soc_sos} involving polynomials $\omega_j, \beta_j \in \R[t, x]_{\leq 2d}$ will have $m$ maximal-size Gram matrices of size $2\binom{n+d}{n}$ from constraint \eqref{eq:robust_soc_sos_sdp}.

\begin{rmk}
\label{rmk:rotated_soc}
    The rotated \ac{SOC} cone is $Q^n_r = \{(u, v, z) \in \R^n \times \R^2_\geq 0 \mid \norm{u}_2^2 \leq v z\}$ \cite{alizadeh2003second}. Membership in $Q^n_r$ may be expressed as a linear transformation of membership in $Q^{n+1}$ by
    \begin{align}
    \label{eq:rotated_soc}
        (u, v, z) \in Q^n_r \Leftrightarrow ([2u, v-z], v+z) \in Q^{n+1}.
    \end{align}
    The identity \eqref{eq:rotated_soc}  can be used to form \ac{SOS}-matrix programs from \eqref{eq:robust_counterpart_soc_sos} for rotated-\ac{SOC} constrained uncertainty sets $W$.
\end{rmk}

\subsection{Approximation Result}

The following theorem summarizes the above restrictions.

\begin{thm}    
\label{thm:sos_robust_tight}
Assume A1-A6. Let the \ac{SDR} cone $K$ from the $W$-representation \eqref{eq:w_sdr} be the product of nonnegative, \ac{SOC}, and \ac{PSD} cones. Then the \ac{SOS} programs derived from \eqref{eq:lie_robust_counterpart} (by example \eqref{eq:robust_lie_polytope_sos},  \eqref{eq:robust_lie_psd_sos}, \eqref{eq:robust_counterpart_soc_sos}) will converge to the strict version of \eqref{eq:lie_expand} when $v$ is polynomial.
\end{thm}
\begin{proof}
    The multiplier functions $\zeta$ may be chosen to be polynomial by Theorem  \ref{thm:robust_zeta_poly}. The Archimedean condition of $[0, T] \times X$ ensures that \ac{SOS}-matrices will generate all positive \ac{PSD} matrices over $[0, T] \times X$. Because a polynomial $\zeta$ exists by \ref{thm:robust_zeta_poly}, it will be found at some finite-degree \ac{SOS} tightening, thus proving the theorem.
\end{proof}




%% file: sections/discrete_time.tex
\section{Discrete-Time Constraints}
\label{sec:discrete_time}

The prior content of this paper involves a continuous-time peak estimation problem \eqref{eq:peak_traj_poly} and robustifies the continuous-time Lie constraint \eqref{eq:lie_robust_counterpart}. This section demonstrates how parameterized robust counterpart theory can be employed to simplify the computation of problems with discrete-time dynamics, such as the discrete-time peak estimation task of
\begin{align}
    P^* = & \sup_{t^* \in 0..T,\, x_0 \in X_0,\, w(\cdot)} p(x[t^* \mid x_0, w[\cdot]]) \label{eq:peak_traj_discrete}\\
    & x[t+1] = f(x[t], w[t]), \ w[t] \in W \quad  \forall t \in 0..T, \nonumber
    & x[0] = x_0.
    \end{align}

\subsection{Discrete-Time Background}

An infinite-dimensional \ac{LP} that solves \eqref{eq:peak_traj_discrete} with $P^*=d^*$ under assumptions A1-A5 is (Equation (20) of \cite{miller2021uncertain})
\begin{subequations}
\label{eq:peak_un_disc}
\begin{align}
    d^* = & \ \inf_{v, \gamma, \alpha} \quad \gamma + T \alpha \label{eq:peak_un_disc_obj}\\
    & {\gamma} \geq {v(x)}  & &   \forall x \in X_0 \label{eq:peak_un_disc_init}\\
    & v(f(x, w)) -v(x)  \leq  \alpha  & & \forall (x, w) \in  X \times W\label{eq:peak_un_disc_flow}\\
    & v(x) \geq  p(x) \label{eq:peak_cont_disc_p}  & & \forall x \in X  \\
    &v(x) \in C(X), \ \gamma \in \R, \ \alpha \geq 0\label{eq:peak_cont_un_disc_v}.
\end{align}
\end{subequations}

Under the imposition that $f$ is disturbance-affine (such as in \eqref{eq:disturbance_affine}), the dynamical constraint \eqref{eq:peak_un_disc_flow} can be expressed as 
\begin{align}
    \alpha + v(x) - v\left(f_0(x) + \textstyle \sum_{\ell=1}^L w_\ell f_\ell(x)\right) \geq 0 & & \forall (x, w) \in  X \times W. \label{eq:disc_affine}
\end{align}

When the auxiliary function $v(x)$ is chosen to be nonlinear in $x$, the left-hand side of \eqref{eq:disc_affine} is no longer disturbance-affine in $w$. This nonlinearity prevents the paramaterized robust counterpart method of Section \ref{sec:lie_decomposed} from directly being used to decompose \eqref{eq:disc_affine}, and was noted as an obstacle in scalable synthesis of discrete-time optimal control laws in \cite{Han_2018}.

\subsection{Robustification through Lifting} 

In order to robustify \eqref{eq:disc_affine}, we will employ a lifting variable $\tilde{x}$ under the constraint that $\tilde{x} = f(x)$. The support set $\Omega$ of variables $(x, \tilde{x}, w)$ is

\begin{align}
    \Omega = \left\{(x, \tilde{x}, w) \in X^2 \times W(x) \mid \tilde{x} = \textstyle f_0(x) + \sum_{\ell=1}^L w_\ell f_\ell(x)\right\}. \label{eq:disc_omega}
\end{align}

The strict version of constraint \eqref{eq:disc_affine} is then equivalent to the lifted term of
\begin{align}
    \alpha + v(x) - v(\tilde{x})  > 0 & & \forall (x, \tilde{x}, u) \in \Omega. \label{eq:gen_lift}
\end{align}

By using Remark \ref{rmk:equality} to represent the equalities constraint of \eqref{eq:disc_omega}, we obtain a correspondence for \eqref{eq:gen_lift} in the form of \eqref{eq:lin_robust} as
\begin{subequations}
\label{eq:disc_correspondence}
\begin{align}
    b_0 &= \alpha + v(x) - v(\tilde{x}) & a_0 &= \tilde{x}-f_0(x)  \\
    b_\ell&= 0 & a_\ell &= -f_\ell(x).
\end{align}
\end{subequations}

The robustification of \eqref{eq:gen_lift} is equivalent to 
    \begin{subequations}
    \label{eq:disc_robust_counterpart}
    \begin{align}
    & \alpha + v(x) - v(\tilde{x}) > b^T \zeta(x, \tilde{x}) + (\tilde{x}-f_0(x))^T \mu(x, \tilde{x})& & \forall (x, \tilde{x}) \in X^2 \label{eq:disc_robust_counterpart_crit} \\
    & G_s^T \zeta_s(x, \tilde{x}) = 0 & & \forall s = 1..N_s \\
    & \textstyle \sum_{s=1}^{N_s} (A_s^T \zeta_s(x, \tilde{x}))_\ell -f_\ell(x)^T \mu(x, \tilde{x}) = 0 & & \forall \ell=1..L\\
     & \zeta_s(x, \tilde{x}) \in K_s^* & & \forall s = 1..N_s, (x, \tilde{x}) \in X^2. \label{eq:disc_robust_zeta} \\
      & \mu_i(x, \tilde{x}) \in \R & & \forall i = 1..n, (x, \tilde{x}) \in X^2.\label{eq:disc_robust_mu}
    \end{align}
    \end{subequations}


Theorem \ref{thm:multipliers_cont_poly} can be used to prove that $(\zeta(x, \tilde{x}), \zeta(x, \tilde{x}))$ have continuous selections under A1-A5 and A5', and polynomial selections if A6' also holds. 

This discrete-time robustification is generally computationally favorable if $2n < L$, given that the maximal size of a Gram matrix involved in \ac{SOS} constraints for \eqref{eq:disc_affine} (in degree $k$) will transform from $\binom{n+L+k}{k}$ to $\binom{2n+k}{k}$.

\subsection{Discrete-Time Robustified Peak Estimation}
We demonstrate the discrete-time robustification procedure on the peak estimation problem in \eqref{eq:peak_un_disc}. In the specific case of polytope-structured uncertainty $W = \{w \in \R^L \mid A w \leq b\}$ with $b \in \R^m$ \eqref{eq:w_set}, the robustification of \eqref{eq:peak_un_disc} using the expression in \eqref{eq:disc_robust_counterpart} is
\begin{subequations}
\label{eq:peak_un_disc_rob}
\begin{align}
    d^* = & \ \inf_{v, \gamma, \alpha} \quad \gamma + T \alpha \label{eq:peak_un_disc_rob_obj}\\
    & {\gamma} \geq {v(x)}  & &   \forall x \in X_0 \label{eq:peak_un_disc_rob_init}\\
&\alpha + v(x) - v(\tilde{x}) > b^T \zeta(x, \tilde{x}) + (\tilde{x} - f_0(x) )^T \mu(x, \tilde{x})& & \forall (x, \tilde{x}) \in X^2 \\
    & A_\ell^T \zeta(x, \tilde{x}) - f_\ell(x)^T \mu(x, \tilde{x}) = 0 & & \forall \ell \in 1..n \\
    & v(x) \geq  p(x) \label{eq:peak_cont_disc_rob_p}  & & \forall x \in X  \\
    &v(x) \in C(X), \ \gamma \in \R, \ \alpha \geq 0\label{eq:peak_cont_un_disc_rob_v} \\
    & \zeta \in (C_+(X^2) )^m, \ \mu \in (C(X^2))^n.
\end{align}
\end{subequations}

\begin{rmk}
    Hybrid systems with uncertainties in \ac{ODE} and reset dynamics can be analyzed and controlled by combining the continuous-time robustification in     \eqref{eq:lie_robust_counterpart}  with the discrete-time robustification \eqref{eq:disc_robust_counterpart} \cite{miller2023hybrid}.
\end{rmk}


%% file: sections/data_driven.tex
\section{Data-Driven Setting}

\label{sec:data_driven_background}

This section reviews the $L_\infty$ bounded noise setting and its derived polytopic input constraints for $W$ \cite{cheng2015robust, dai2018moments}. We note that other input sets in a set-membership-based data-driven framework include elementwise $L_1$ noise (sparse channel disturbances), elementwise $L_2$ noise \cite{martin2022gaussian} (e.g., Chi-squared chance constraints on a Gaussian distribution), and semidefinite energy-bounded-noise \cite{waarde2020noisy}.

Samples $y$ of an unknown continuous-time system $\dot{x} = F(t, x)$ are observed according to the relation  $\dot{x}_{observed} = y = F(t, x) + \eta$ with noise term $\norm{\eta}_\infty \leq \epsilon_w$.
The ground-truth system $F(t, x)$ is represented by an affine combination 
\begin{equation}
    \textstyle \dot{x} =F(t, x)=f(t, x; w) = f_0(t, x) + \sum_{\ell=1}^L w_\ell f_\ell(t, x) \label{eq:parameter_data_driven}
\end{equation}
where the parameters $\{w_\ell\}_{\ell=1}^L$ are \textit{a-priori} unknown. The function $f_0$ represents prior knowledge of system dynamics $F$, and the dictionary functions $\{f_\ell\}$ serve to describe unknown dynamics.

The tuples $\Dc_k = (t_k, x_k,  y_k)$ for $k = 1..N_s$ observations are contained in the data $\Dc$. System parameters $w$ that are consistent with data in $\Dc$ form a set
 \begin{equation}
\label{eq:theta_x_set}
    W = \{ w \in \R^L \mid \forall k: \norm{y_k - f(t_k, x_k; w)}_\infty \leq \epsilon_w \}.
\end{equation}

The set $W$  from \eqref{eq:theta_x_set} may be described in terms of matrices
\begin{subequations}
\label{eq:poly_new_terms}
\begin{align}
\Gamma_{ik\ell} &=  [-f_{i\ell}(t_k, x_k); f_{i\ell}(t_k, x_k)] \\
h_{ik} &= \begin{bmatrix} - y_{ik} + f_{i0}(t_k, x_k) \\
 y_{ik} - f_{i0}(t_k, x_k)\end{bmatrix}.
\end{align}
\end{subequations}
as the polytope
\begin{equation}
\label{eq:poly_simple}
    W = \left\{w \in \R^L \middle| \forall i,k: \   \textstyle \sum_{\ell=1}^L \Gamma_{ik\ell} w_\ell \leq \epsilon_w + h_{ik}\right\}.
\end{equation}
The expression in \eqref{eq:poly_simple} will be written concisely as the polytope $W = \{w \mid \Gamma w \leq (\epsilon_w + h$)\}.

\begin{rmk}
The compactness and non-emptiness assumption of A4 is satisfied when the $L_\infty$ bound $\epsilon_w$ is finite and sufficiently many observations in $\Dc$ are acquired.
\end{rmk}

The set $W$ as described in \eqref{eq:poly_simple} has $m = 2 N_s n$ affine constraints, most of which are redundant. These redundant constraints can be identified and dropped through the \ac{LP} method of \cite{caron1989degenerate}.
The multiplier term $\zeta$ is $m$-dimensional, so lowering $m$ by eliminating redundant constraints is essential in creating tractable problems.

\begin{rmk}
\label{rmk:data_conservative} 
    The parameters $w$ of the true system $F(t, x)$ from \eqref{eq:parameter_data_driven} are constant in time. This constancy may be implemented by treating $w$ as new states with $\dot{w}(t) = 0$. Unfortunately, the augmentation of new states would require auxiliary functions of the form $v(t, x, w)$, which would result in non-affine expressions in $w$. The robust-counterpart-based method for data-driven systems analysis relaxes $w$ to become a time-dependent uncertainty process with $w(t) \in W$, allowing for tractable computation (by $w$-elimination) at the cost of conservatism.
\end{rmk}

%% file: sections_arxiv/examples.tex
\section{Examples}
\label{sec:robust_examples}
Code to execute robust counterparts for analysis and control and to replicate figures and experiments is available at 
\url{https://github.com/Jarmill/data_driven_occ}. All source code was developed in MATLAB 2021a. Dependencies include YALMIP \cite{Lofberg2004} to form the \acp{SDP} and MOSEK \cite{mosek92} to solve them. Unless otherwise specified, the \ac{SDR} uncertainty set $W$ will be polytopic. Redundant constraints in the polytope $W$ were identified and dropped through the \ac{LP} method of \cite{caron1989degenerate}. 
When $W$ is polytopic, the $w(t)$ inputs of trajectory samples (data-driven analysis) were acquired through hit-and-run sampling \cite{kroese2013handbook} as implemented by \cite{benham2021}. In the case of semidefinite-bounded noise, the input $w(t)$ was chosen by choosing a uniformly random direction $\theta$ on the $(L-1)$-sphere and solving  the \ac{LMI} $\max_{w \in W}\theta^T w$.

\subsection{Elliptope-Disturbed  System}
This subsection performs peak estimation of a cubic system under a semidefinite-constrained disturbance process (modified from \cite{prajna2004safety}):
\begin{subequations}
\label{eq:flow_pillow}
\begin{align}
    f(t, x, w) &= \begin{bmatrix} x_2 \\ -x_1 -x_2 + x_1^3/3 + w_1 x_1 + w_2 x_1 x_2 + w_3 x_3        
    \end{bmatrix} \\
    W &= \left\{w \in \R^3:\begin{bmatrix}
        1 & w_1 & w_2 \\ w_1 & 1 & w_3 \\ w_2 & w_3 & 1
    \end{bmatrix} \in \psd^3_+ \right\}. \label{eq:elliptope}
\end{align}
\end{subequations}

The set in \eqref{eq:elliptope} is the standard convex elliptope/pillow spectahedron.
Dynamics in \eqref{eq:flow_pillow} start at $X_0 = [1.25; 0]$ and continue for $T=5$ time units in the space $X = [-0.5, 1.75]\times  [-1, 0.5]$. The first 6 bounds of maximizing $p(x) = -x_2$ along these trajectories, after performing a robust counterpart, are $p^*_{1:6} = [1, 1, 0.8952, 0.8477, 0.8471, 0.8470]$. At order $6$, the largest \ac{PSD} matrix constraint (for the $3 \times 3$ \ac{SOS}-matrix) is of size $3\binom{3 + 6}{6} = 252$ in the variables $(t, x)$ after performing robust decomposition, while the non-decomposed largest \ac{PSD} size is $3\binom{6 + 6}{6} = 2772$ in the variables $(t, x, w)$. Sample trajectories of \eqref{eq:flow_pillow} are plotted in Figure \ref{fig:flow_pillow}.

\begin{figure}[!h]
    \centering
    \includegraphics[width=0.5\linewidth]{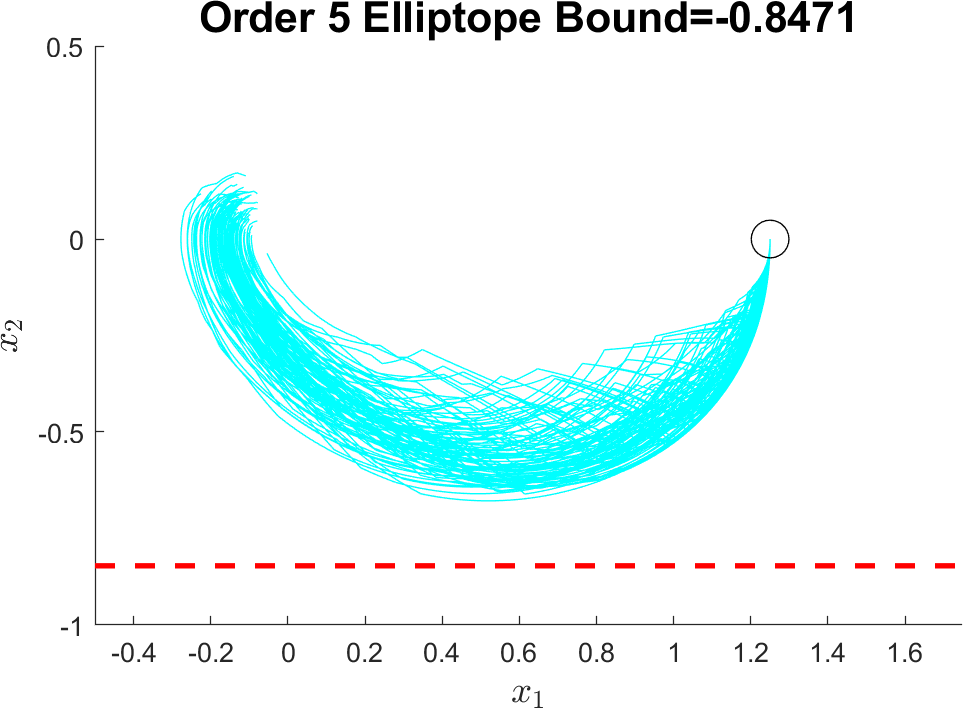}
    \caption{Order-5 bound on minimal $x_2$ for Flow \eqref{eq:flow_pillow} under elliptope-constrained noise}
    \label{fig:flow_pillow}
\end{figure}
\subsection{Data-Driven Flow System}

\label{sec:flow_poly}

The uncontrolled Flow system from \cite{prajna2004safety} is
\begin{align}
\label{eq:flow}
    \dot{x}(t) &= [x_2(t); \ -x_1(t) -x_2(t) + x^3_1(t)/3]  \qquad \forall t' \in [0, 5].
\end{align}

This subsection will perform data-driven peak and distance estimation. The derivative $\dot{x_2}$ is modeled by a cubic polynomial $\dot{x_2} = \sum_{\deg \alpha \leq 3} w_\alpha x_1^{\alpha_1} x_2^{\alpha_2}$ with 10 unknown parameters/inputs $\{w_\alpha\}$. The derivative $\dot{x}_1 = x_2$ remains known and there is no uncertainty in the first coordinate. The system model for the data-driven flow system is
\begin{align}
\label{eq:flow_data_driven}
    \dot{x}(t) &= [x_2(t); \ \textstyle \sum_{\deg \alpha \leq 3} w_\alpha(t) x_1^{\alpha_1}(t) x_2^{\alpha_2}(t)]  \qquad \forall t' \in [0, 5].
\end{align}


Figure \ref{fig:flow_observed_low_eps} visualizes $N=40$ observed data points sampled within the initial set $X_{sample} = \{x \mid (x_1-1.5)^2 + x_2^2 \leq 0.4^2\}$. The true derivative  values are the blue arrows and the $\epsilon=[0; 0.5]$-corrupted derivative observations are orange. The $N=40$ points yield $2N=80$ affine constraints, of which the polytope $W$ has $L=33$ faces (nonredundant constraints) and 7534 vertices. 

Figure \ref{fig:flow_poly_all} displays  system trajectories of \eqref{eq:flow_data_driven} for a time horizon of $T=5$ starting from the point $X_0 = (1.5, 0)$ (left, Figure \ref{fig:flow_poly_pt}) and from the circle $X_0=X_{\text{sample}}$ (right, Figure \ref{fig:flow_poly_circ}), when the uncertainty process $w(t)$ is restricted to $W$. Each case desires to maximize $p(x)=-x_2$ over the state region of $X = \{x \mid \norm{x}_2^2 \leq 8\}$.
The first 4 \ac{SOS} peak estimates in the point $X_0$ case (Figure \ref{fig:flow_poly_pt}) are $d_{1:4}^* = [2.828, 2.448, 1.018, 0.8407]$. The first four estimates in the disc $X_0$ case (Figure \ref{fig:flow_poly_circ}) are $d_{1:4}^* = [2.828, 2.557, 1.245, 0.894]$. 


\begin{figure}[ht]
    \centering
    \includegraphics[width=0.4\linewidth]{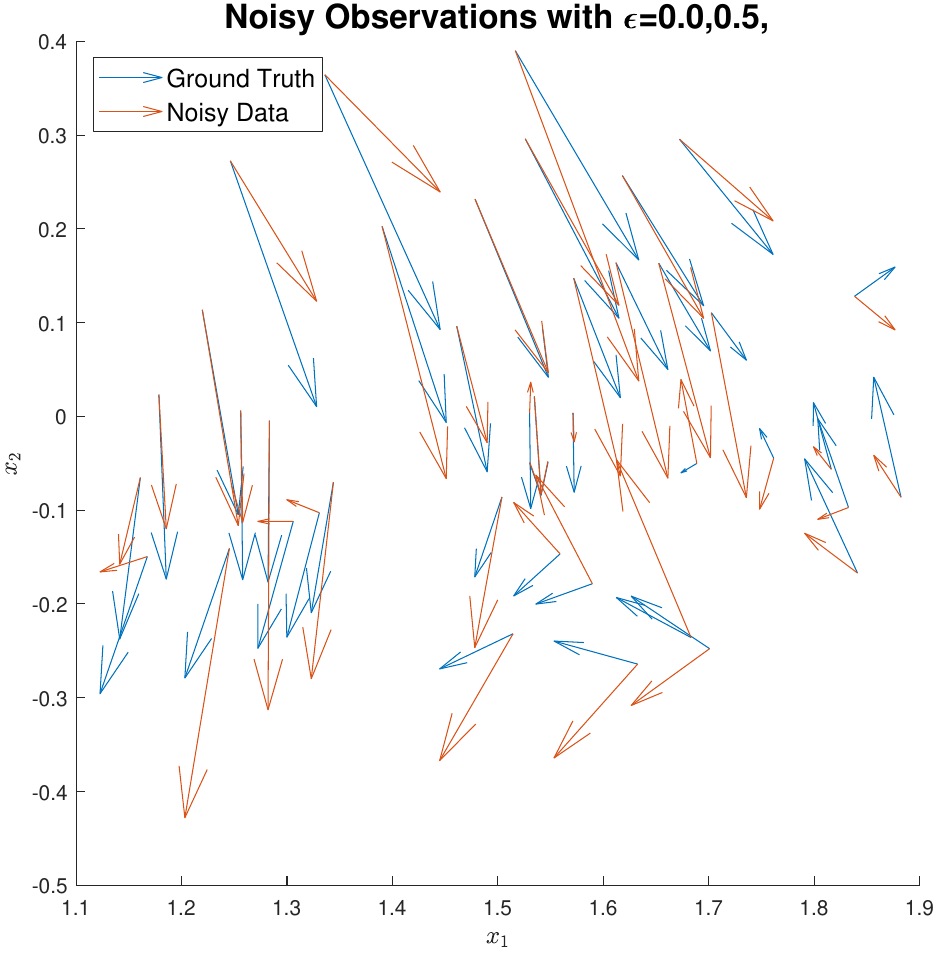}
    \caption{\label{fig:flow_observed_low_eps}Observed data of Flow system \eqref{eq:flow} within a circle}
\end{figure}

    \begin{figure}[ht]    
     \centering
     \begin{subfigure}[b]{0.48\linewidth}
         \centering
         \includegraphics[width=\linewidth]{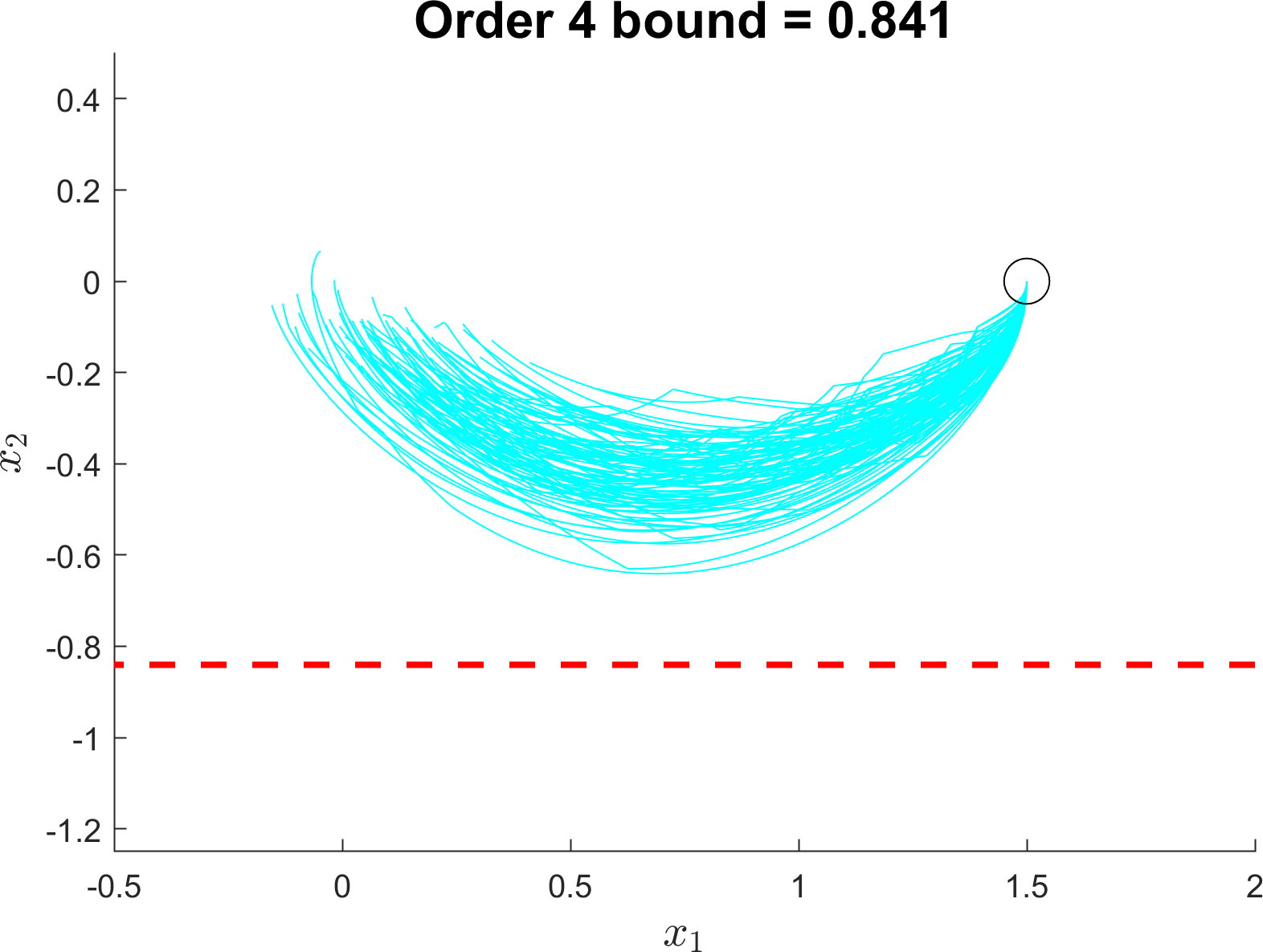}
         \caption{\label{fig:flow_poly_pt} $X_0 = (1.5, 0)$}
         
     \end{subfigure}
     \;
     \begin{subfigure}[b]{0.48\linewidth}
         \centering
         \includegraphics[width=\linewidth]{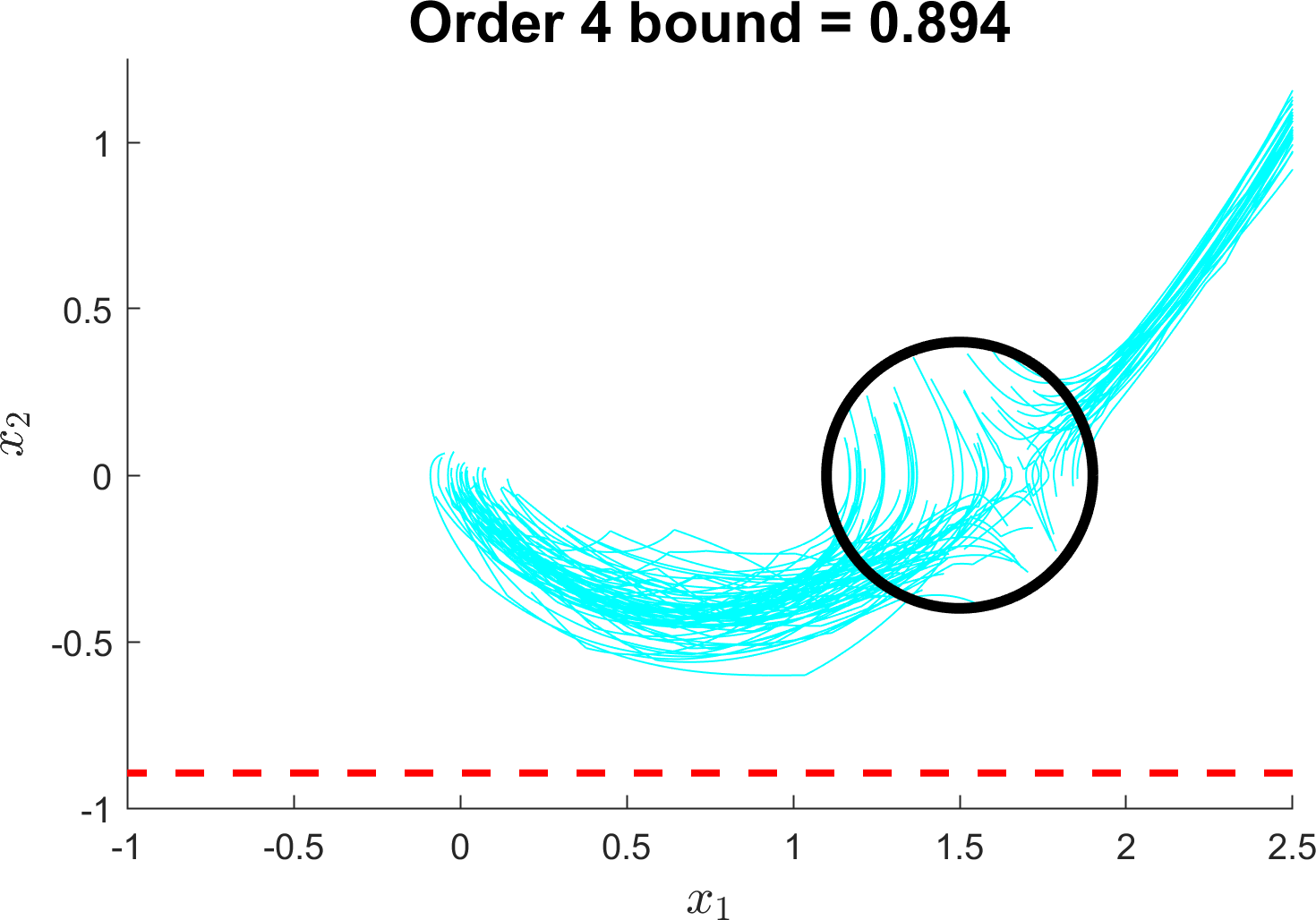}
         \caption{\label{fig:flow_poly_circ}$X_0 =X_{\text{sample}}$}
     \end{subfigure}
      \caption{\label{fig:flow_poly_all} Minimizing $x_2$ on the observed Flow system \eqref{eq:flow_data_driven} at order-4 SOS tightening}
\end{figure}

Figure \ref{fig:dist_flow_data} displays the result of distance estimation on the Flow system \eqref{eq:flow_data_driven} with $N=40$ points and $\epsilon = [0; 0.5]$. 
The initial point is $X_0 = [1; 0]$ in the state set $X = [-1, 1.25] \times [-1.25, 0.7]$, and trajectories are tracked for $T=5$ time units. The distance function is the $L_2$ distance and the red half-circle unsafe set is $X_u = \{x \mid 0.5^2 \geq - (x_1+0.25)^2 - (x_2+0.7)^2, \-(x_1+0.25)/\sqrt{2} +(x_2+0.7)^2/\sqrt{2} >=0\}$. The first 5 bounds of the robust distance estimation program are $c^*_{1:5} = [1.698\times 10^{-5}, \ 0.1936, \ 0.2003, \ 0.2009, \ 0.2013]$.

\begin{figure}[h]
    \centering
    \includegraphics[width=0.5\linewidth]{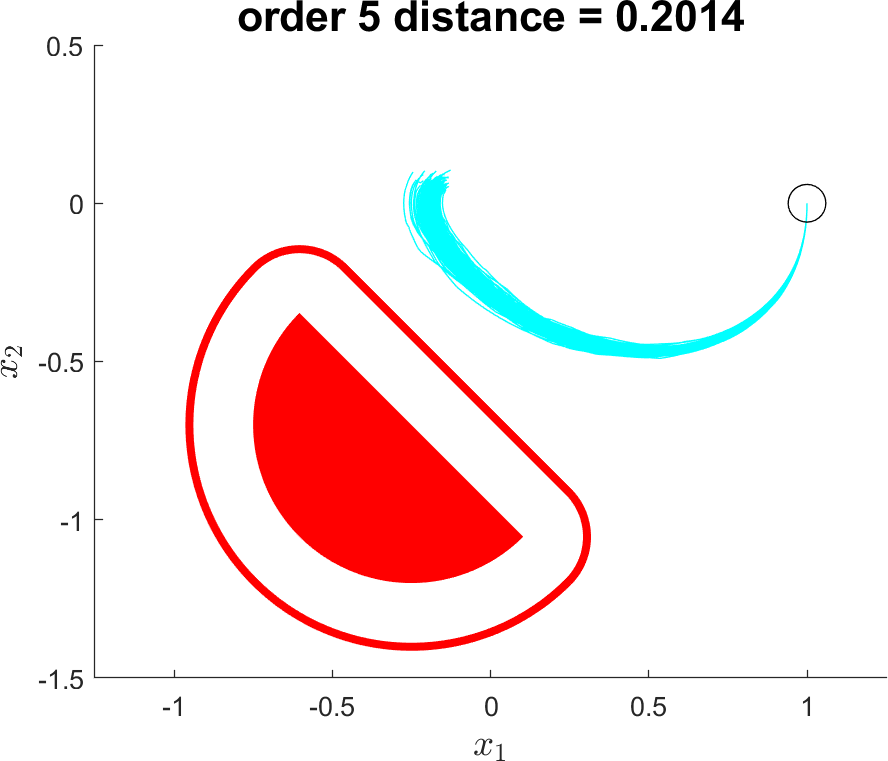}
    \caption{Distance estimate of \eqref{eq:flow} at order 5}
    \label{fig:dist_flow_data}
\end{figure}

\subsection{Twist System}

This section performs peak estimation on the Twist system from \cite{miller2021distance}
\begin{align}
    \label{eq:twist_dynamics}
    \dot{x}_i(t) &= \textstyle \sum_j B^1_{ij} x_j(t) - B^3_{ij}(4x_j^3(t) - 3x_j(t))/2,
    \end{align}
    \begin{align}
    \label{eq:twist_parameters}
    B^1 &= \begin{bmatrix}-1 & 1 & 1\\ -1 &0 &-1\\ 0 & 1 &-2\end{bmatrix} &  B^3 &=  \begin{bmatrix}-1 & 0 & -1\\ 0 &1 &1\\ 1 & 1 &0\end{bmatrix}.
       \end{align}
Choosing different parameter matrices $B^1$ (linear) and $B^3$ (cubic) yields a family of dynamical systems, many of which are attractors and some of which possess limit cycles. We note that the Twist system possesses a symmetry $x \leftrightarrow -x$, and therefore only the top portion with $x_3 \geq 0$ will be treated.
A total of $N=100$ observations with a noise bound of $\epsilon=0.5$ are taken, and are plotted in Figure \ref{fig:twist_observed}. These $N=100$ observations will induce $2Nn = 600$ affine constraints on eventual polytopes $W$.

Details of the auxiliary function \acp{LP} for distance estimation, reachable set estimation, and \ac{ROA} maximization are found in Appendix \ref{app:analysis_control_problem}. The decomposable Lie constraints in these problems (from Appendix \ref{app:analysis_control_problem}) are \eqref{eq:poly_dist_w_lie}, \eqref{eq:reach_cont_lie}, and \eqref{eq:roa_cont_lie} respectively.

\begin{figure}[!ht]
         \centering
         \includegraphics[width=0.4\linewidth]{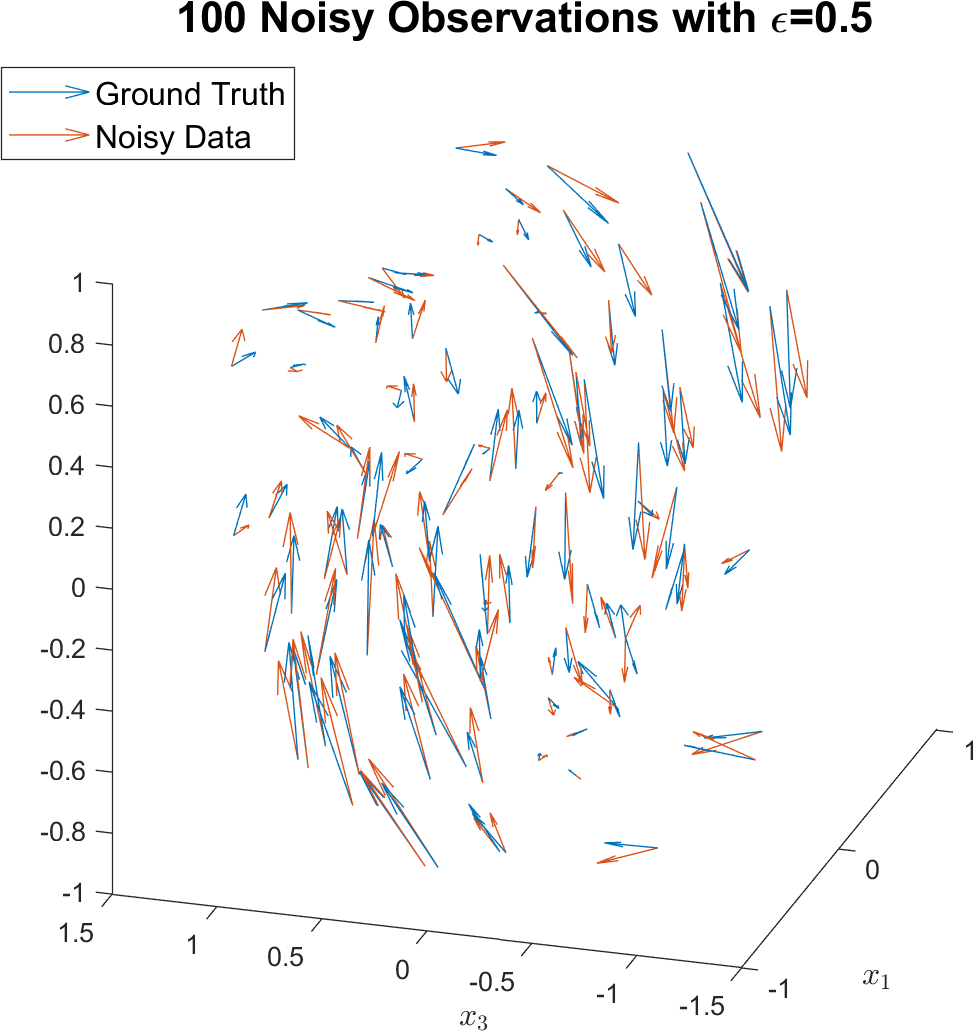}
         \caption{\label{fig:twist_observed}100 observations of Twist system \eqref{eq:twist_dynamics}}
\end{figure}

All scenarios in this subsection will find peak estimates on the maximum value of $p(x)=x_3$ of the Twist system over the space $X = \{x \mid -1 \leq x_1, x_2 \leq 1, \ 0 \leq x_3 \leq 1\}$ and time horizon $T=8$ starting at $X_0=[-1,0,0]$.

    
    Figure \ref{fig:twist_one_unknown} involves the $L=9$ case where $B^1$ is unknown (left) or when $B^3$ is unknown (right). 
    When $B^1$ is unknown, the matrix $B^1$ in \eqref{eq:twist_parameters} is replaced by a matrix of parameter inputs $w: [0, T] \rightarrow \R^{3 \times 3}$ according to Remark \ref{rmk:data_conservative} (with a similar substitution in the unknown $B^3$ case).
    
    The unknown $B^1$ case in Figure \ref{fig:twist_a} has a polytope $W$ with $m=30$ faces and peak bounds of $d_{1:3}^* = [1.000, 0.9050, 0.8174]$. The known $B^3$ case in Figure \ref{fig:twist_b} also has $m=30$ faces in its polytope with peak bounds of $d_{1:3}^* = [1.000, 0.9050, 0.8174]$. The maximal \ac{PSD} matrix size of the Lie nonpositivity constraint is 2380 pre-decomposition and 70 post-decomposition. Computational limits of the experimental platform restricted the computation of $d^*$ to maximal order 3 .

\begin{figure}[ht]
     \centering
     \begin{subfigure}[b]{0.48\linewidth}
         \centering
         \includegraphics[width=0.7\linewidth]{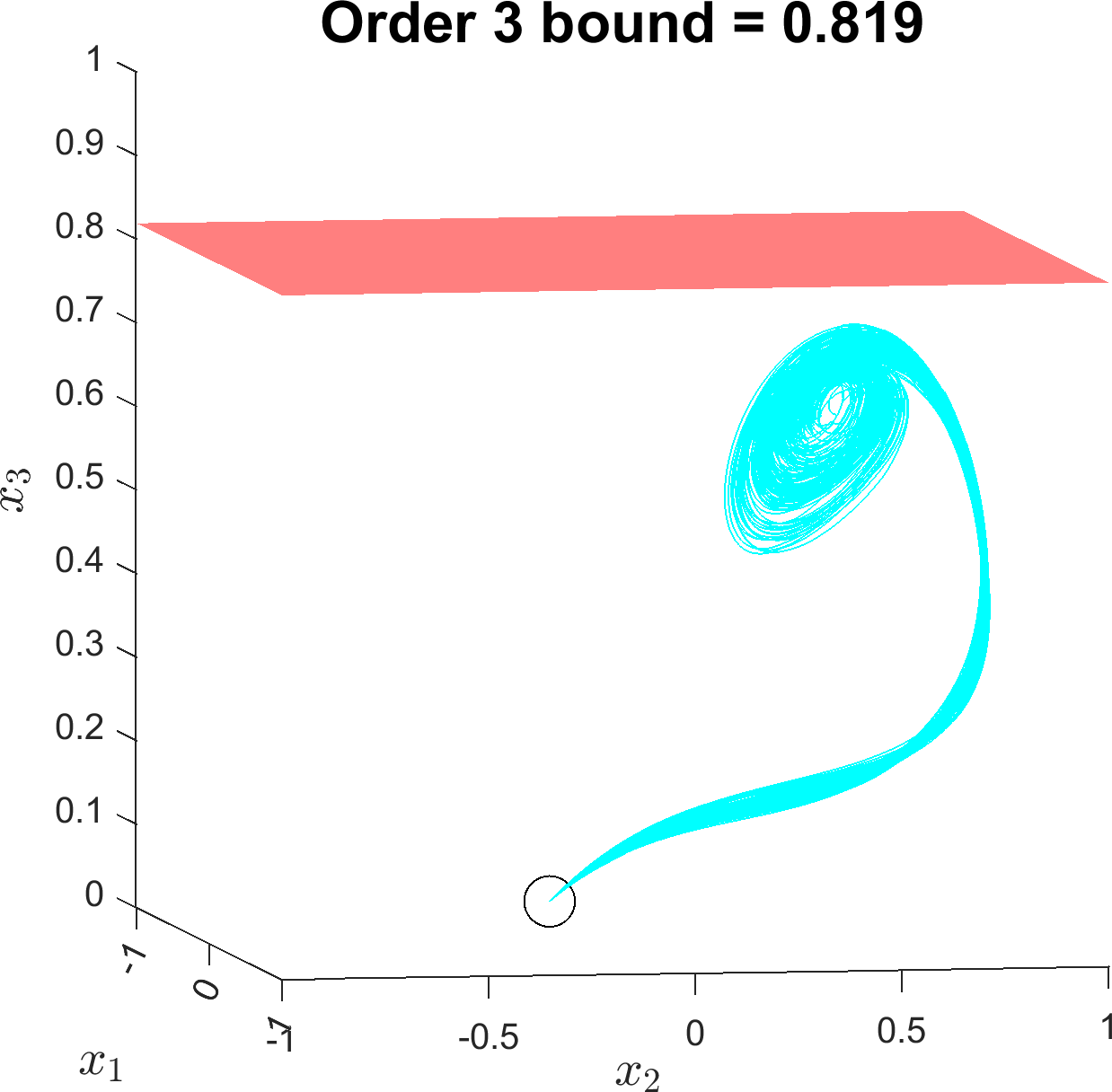}
         \caption{\label{fig:twist_a}Unknown $B^1$, Known $B^3$}
         
     \end{subfigure}
     \;
     \begin{subfigure}[b]{0.48\linewidth}
         \centering
         \includegraphics[width=0.7\linewidth]{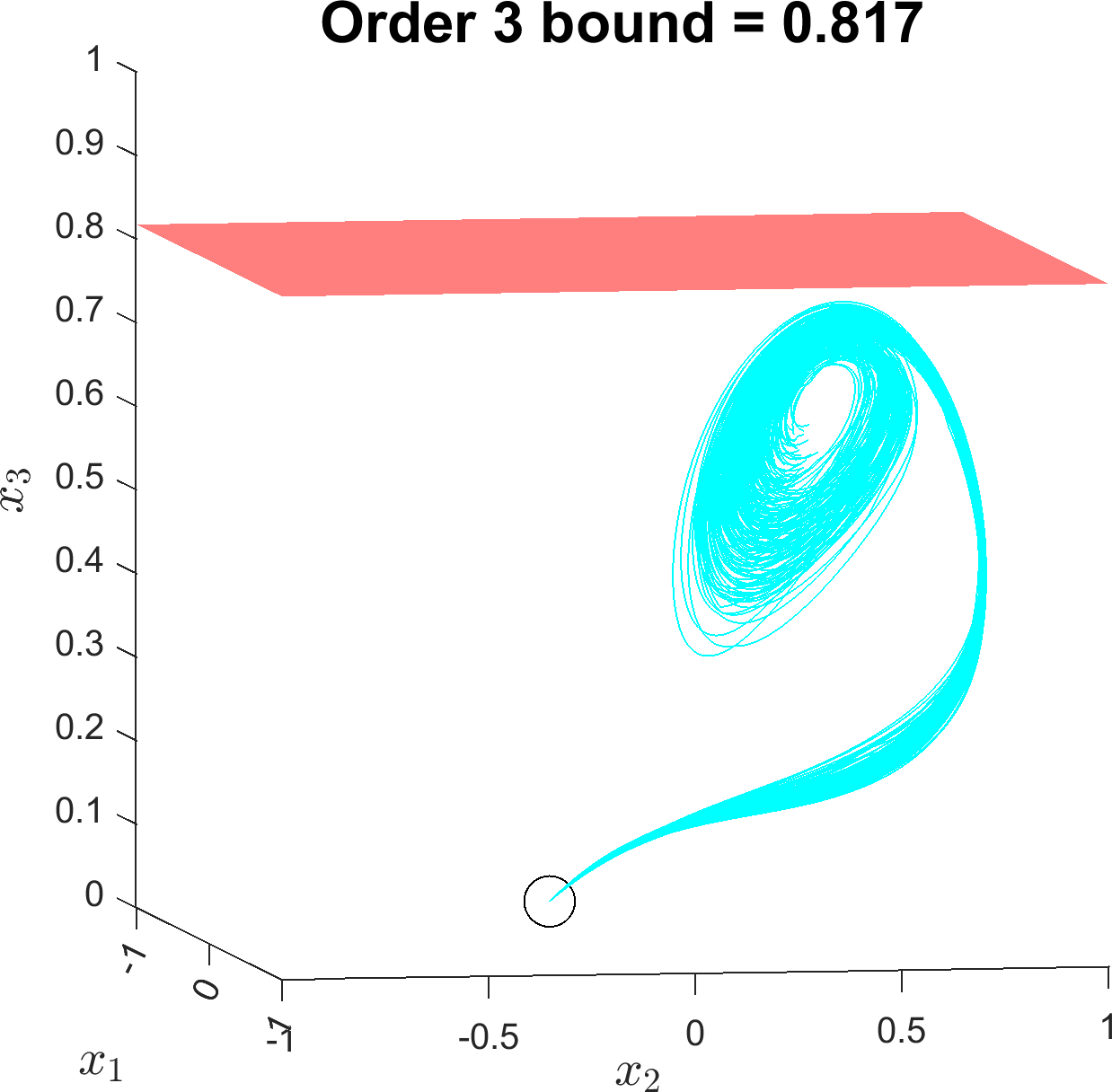}
         \caption{\label{fig:twist_b}Known $B^1$, Unknown $B^3$}
     \end{subfigure}
      \caption{\label{fig:twist_one_unknown}Twist \eqref{eq:twist_dynamics} system where either $B^1$ or $B^3$ are unknown.}
\end{figure}

When both parameters $(B^1, B^3)$ are unknown, the polytope $W$ has $L=18$ dimensions and $m=70$ nonredundant faces. The first peak estimates on this system are $d_{1:2}^* = [1.000, 0.9703]$ as plotted in Figure \ref{fig:twist_both}. At degree 2, the maximal \ac{PSD} matrix size in the Lie constraint falls from 2300 pre-decomposition and to 35 post-decomposition. The experimental platform became unresponsive in YALMIP when attempting to compile the $d^*_3$ model when $(B^1, B^3)$ are unknown.


     \begin{figure}[h]
         \centering
         \includegraphics[width=0.4\linewidth]{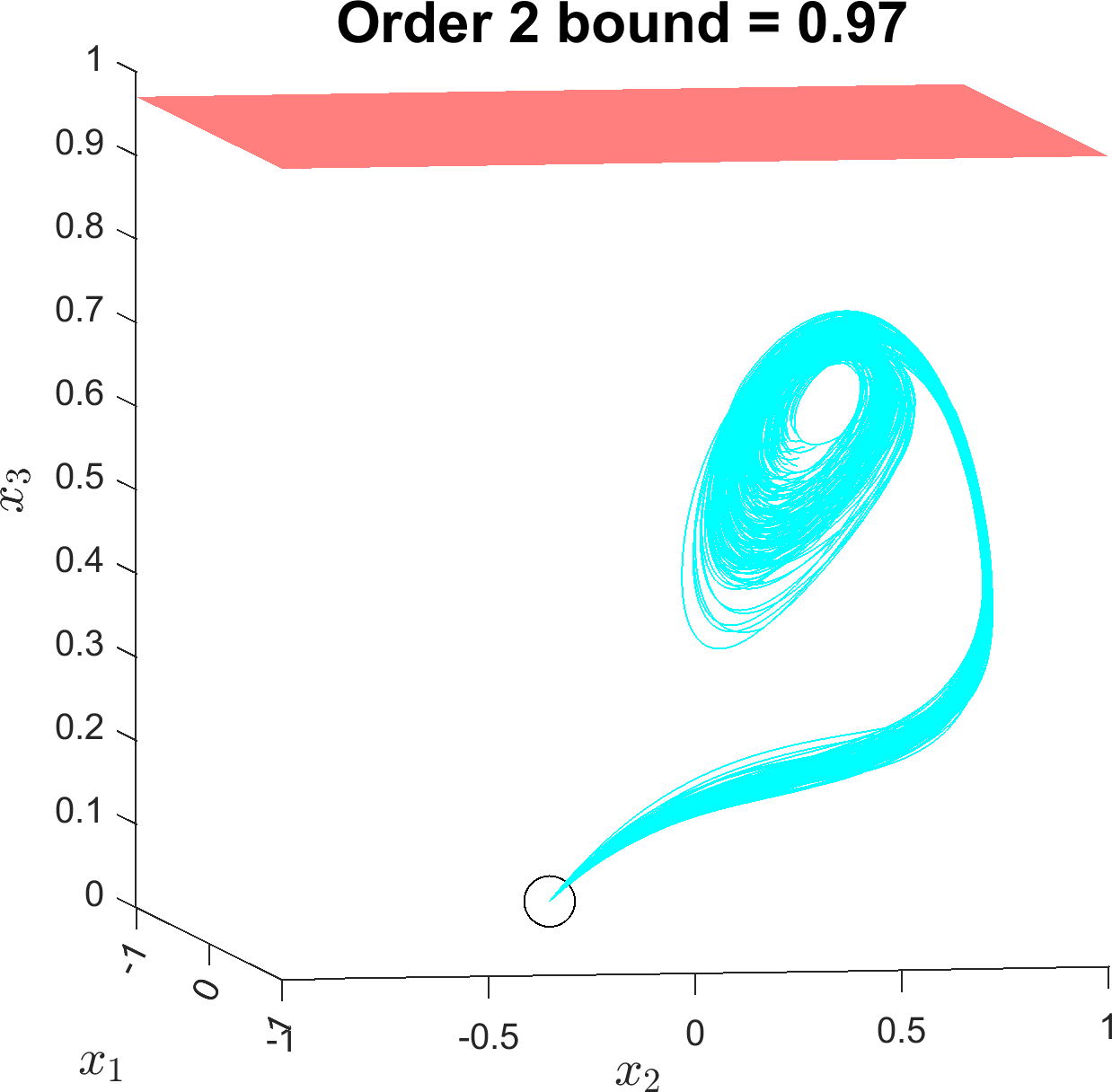}
         \caption{\label{fig:twist_both}Twist \eqref{eq:twist_dynamics} with unknown $(B^1, B^3)$}
     \end{figure}

%% file: sections/discrete_time_example.tex
\subsection{Discrete-Time Peak Estimation Example}

This example performs peak estimation \eqref{eq:peak_traj_discrete} on a discrete-time system.
The ground-truth system is ((35a) from \cite{miller2021uncertain})
\begin{equation}
    x_+ = \begin{bmatrix}-0.3x_1 + 0.8 x_2 + 0.1 x_1 x_2 \\ -0.75 x_1 - 0.3 x_2 +w \end{bmatrix}  \label{eq:disc_dyn_ex}
\end{equation}

The system model for \eqref{eq:disc_dyn_ex} is a quadratic polynomial vector with unknown coefficients $(w^1, w^2)$:
\begin{equation}
    x_+ = \begin{bmatrix}\sum_{\deg \alpha \leq 2} w^1_\alpha x^{\alpha_1}_1 x^{\alpha_2}_2, &  \  \sum_{\deg \alpha \leq 2} w^2_\alpha x^{\alpha_1}_1 x^{\alpha_2}_2 \end{bmatrix}^T  \label{eq:disc_dyn_ex_model}
\end{equation}

The model in \eqref{eq:disc_dyn_ex_model} has $L=12$ unknown parameters $(w^1, w^2)$. Sampling over $X = [-2, 2]^2$ occurs $N_s = 40$ times, yielding observations corrupted by $L_\infty$-bounded noise ($\epsilon = 0.5$). 

The observed data for state transitions are pictured in Figure \ref{fig:discrete_transition}.
It is desired to bound the maximum value of $p(x) = -x_2$ along discrete-time system trajectories up to a time horizon of $T = 10$.
With a starting point of $X_0 = [-1.5, 0]$,  \ac{SOS} tightenings of the discrete-time peak estimation program in \eqref{eq:peak_un_disc_rob} produce bounds of $d^*_{1:4} = [2.0000, 1.7221, 1.1254, 0.9511]$. The order-4 bound is pictured in Figure \ref{fig:discrete_state}.
 \begin{figure}[ht]
     \centering
     \begin{subfigure}[b]{0.48\linewidth}
         \centering
         \includegraphics[width=\linewidth]{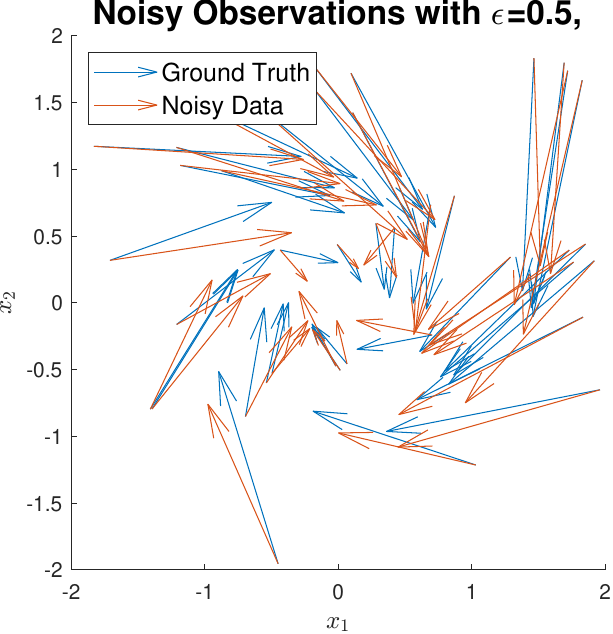}
         \caption{Observed data of \label{fig:discrete_transition} \eqref{eq:disc_dyn_ex}}
         
     \end{subfigure}
     \;
     \begin{subfigure}[b]{0.48\linewidth}
         \centering
         \includegraphics[width=\linewidth]{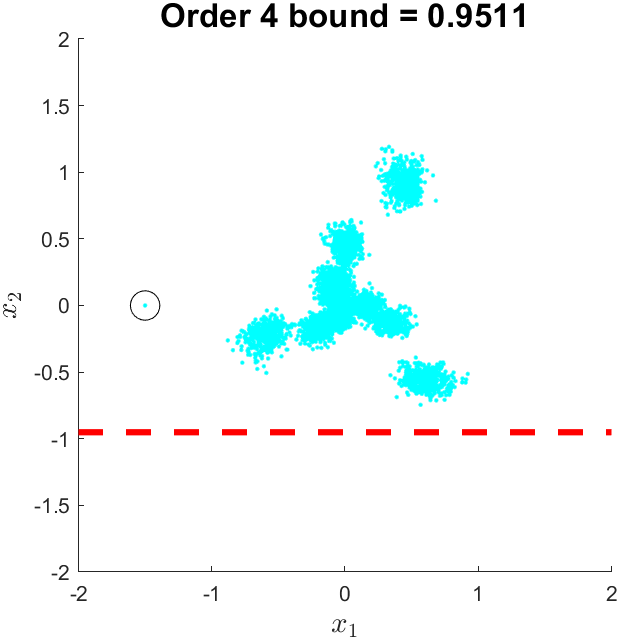}
         \caption{  \label{fig:discrete_state}Order 4 tightening}
     \end{subfigure}
      \caption{\label{fig:discrete_peak}Peak  $p(x)=-x_2$ estimation of  \eqref{fig:discrete_transition} for system model \eqref{eq:disc_dyn_ex_model} with $L_\infty$: $\epsilon=0.5$ corrupted data}
\end{figure}



%% file: sections_arxiv/reach_example.tex
\subsection{Reachable Set Example}

Figure \ref{fig:twist_reach} illustrates data-driven reachable set estimation on the Twist system from \eqref{eq:twist_dynamics} for a time horizon of $T=8$ by \ac{SOS} tightening to the Lie-robustified \eqref{eq:reach_cont}. The 100 observations from this system are pictured in Figure \ref{fig:twist_observed}, yielding a $L=9$-dimensional polytope with 34 non-redundant faces. As the order of tightening to program  increases from 3 to 4, the red region (level set of $\omega(x)$) tightens to the spiraling attractor region of the $T=8$ reachable set. The reachable set computation problem from \eqref{eq:reach_cont} involves auxiliary  functions $v(t, x)$ and $\phi(x)$, such that $\phi(x) \geq 0$ over $X$ and $\phi(x) \geq 1$ over the true $T$-reachable set \eqref{eq:reach_true_set}.
The `volume' in the plot titles is not the true volume of the superlevel set $\{x \mid \phi(x) \geq 1\}$, but is instead the integration of the polynomial $\phi(x)$ over the Lebesgue measure as $\int_X \phi(x) dx$.

 \begin{figure}[ht]
     \centering
     \begin{subfigure}[b]{0.48\linewidth}
         \centering
         \includegraphics[width=\linewidth]{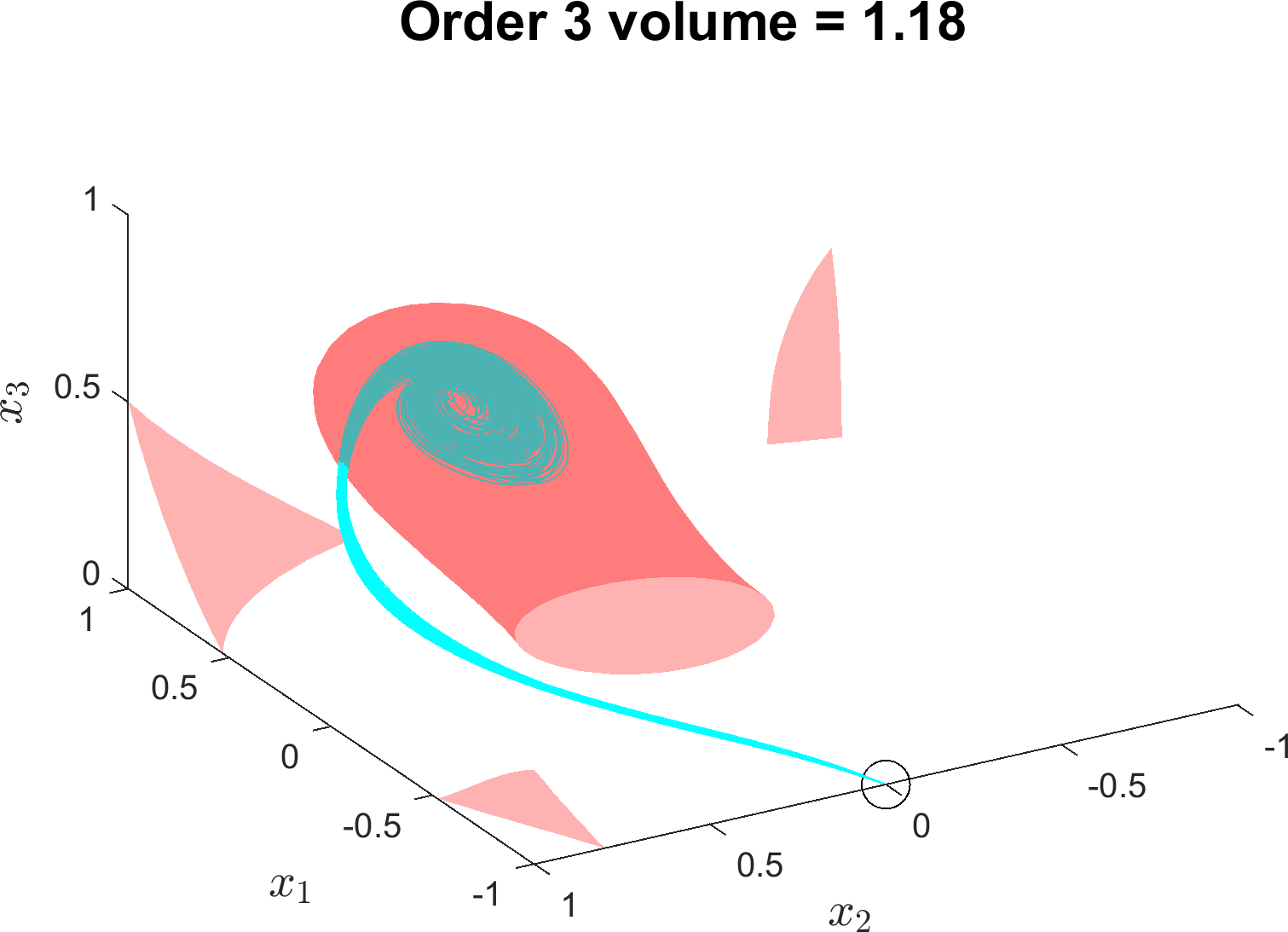}
         \caption{\label{fig:twist_reach_3}Order 3 tightening}
         
     \end{subfigure}
     \;
     \begin{subfigure}[b]{0.48\linewidth}
         \centering
         \includegraphics[width=\linewidth]{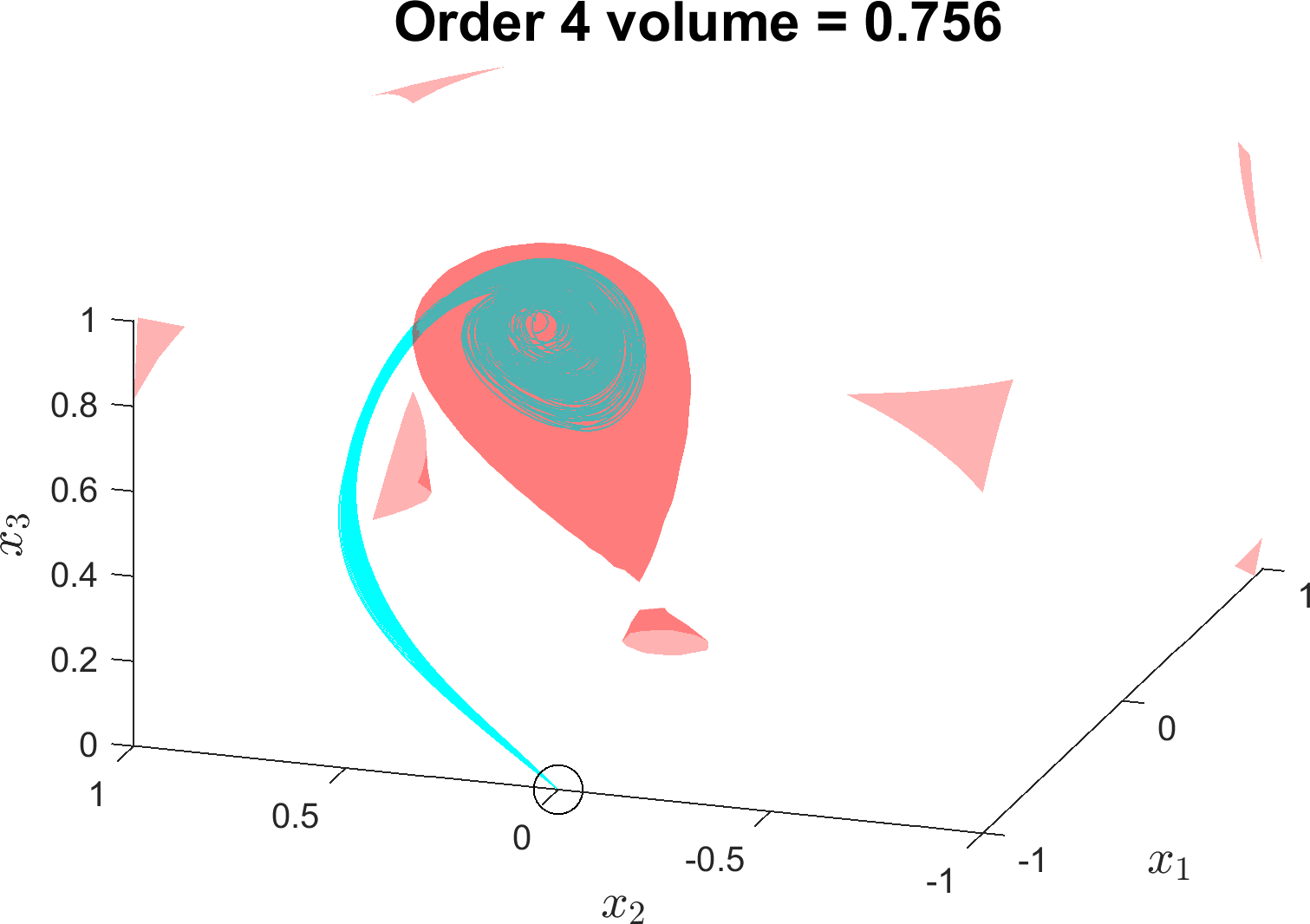}
         \caption{\label{fig:twist_reach_4}Order 4 tightening}
     \end{subfigure}
      \caption{\label{fig:twist_reach}Reachable set estimation of twist \eqref{eq:twist_dynamics} system where $B^1$ is known and $B^3$ are unknown.}
\end{figure}

%% file: sections_arxiv/roa_example.tex
\subsection{Region of Attraction Example}

This example of \ac{ROA} maximization will concentrate on a controlled version of the Flow dynamics from \cite{prajna2004safety} under $L=6$ inputs
\begin{equation}
\label{eq:flow_controlled}
    \dot{x} = \begin{bmatrix}
        x_2(t) \\  -x_1(t) -x_2(t) + x^3_1(t)/3
    \end{bmatrix} + \begin{bmatrix} w_1 + w_2 x_1 + w_3 x_2 \\
    w_4 + w_5 x_1 + w_6 x_2\end{bmatrix}
\end{equation}
obeying the polytopic input limits
\begin{equation}
    W = \left\{w \mid \begin{matrix}\norm{[w_1; w_4]}_\infty \leq 0.1, \ \norm{[w_2; w_3; w_5; w_6]}_\infty \leq 0.15 \\ \norm{[(w_1+w_2+w_3); (w_4+w_5+w_6)]}_\infty \leq 0.3 \end{matrix} \right\}.
\end{equation}

The circle $X_T = \{x \mid 0.1^2 - (x_1-0.5)^2 - (x_2-0.5)^2 \geq 0\}$ is the destination of the \ac{ROA} problem with a time horizon of $T = 5$ and a state space of $X = [-1.5, 1.5]^2$. The \ac{WSOS} tightening of problem the Lie-robustified  \eqref{eq:roa_cont} yields bounds for the \ac{ROA} volume of $d^*_{2:6} =[9.000, \ 9.000, \ 6.717, \ 5.620, \ 5.187].$

The destination set $X_T$ is drawn in the black circle in the left subplot of Figure \ref{fig:flow_roa}. The white area is an outer approximation of the true \ac{ROA}, found as the superlevel set $\{x \mid \phi(x) \geq 1\}$ at order 6. The red area is the sublevel set $\{x \mid \phi(x) \leq 1\}$. The right subplot of Figure \ref{fig:flow_roa} draws the degree-12 polynomial function $\phi(x)$.



\begin{figure}[ht]
         \centering
     \includegraphics[width=\linewidth]{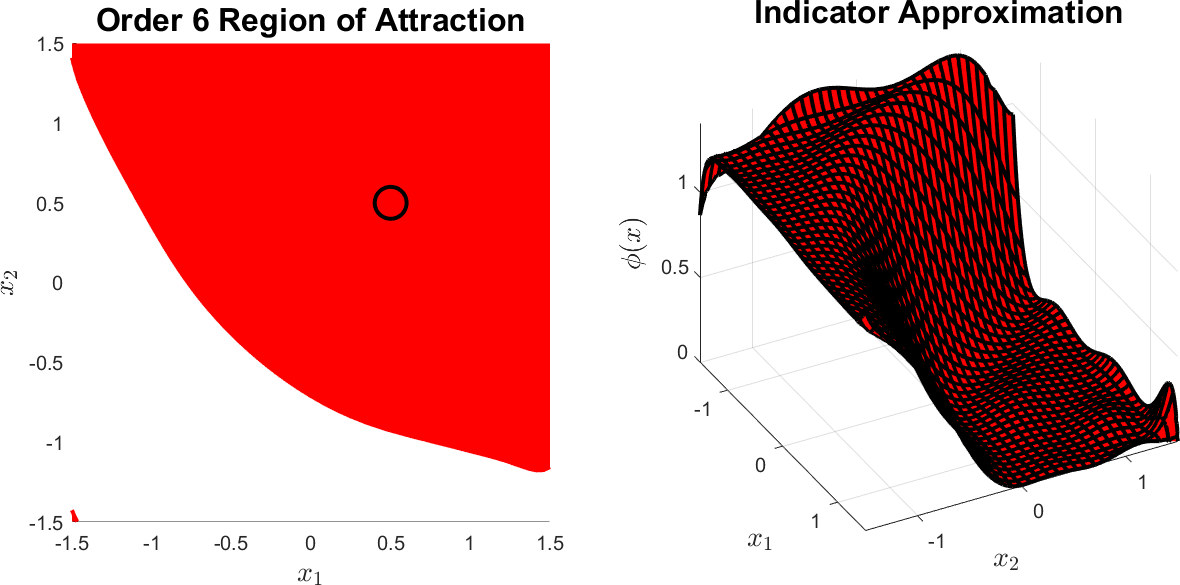}
      \caption{\label{fig:flow_roa} Order 6 \ac{ROA} for controlled Flow \eqref{eq:flow_controlled}}
\end{figure}

%% file: sections/conclusion.tex
\section{Conclusion}
\label{sec:poly_conclusion}

This work formulated infinite-dimensional robust counterparts decomposing input-affine Lie constraints in analysis and control problems. These robust counterparts may be approximated by continuous multipliers without conservatism under compactness and regularity conditions. Elimination of the noise variables $w$ allows for solution and analysis of formerly intractable problems using the moment-\ac{SOS} hierarchy. The robust counterpart method was demonstrated on peak estimation, distance estimation, reachable set estimation, and \ac{BRS}-maximizing control problems. 
Another environment in which these robust counterparts may be employed is in data-driven systems analysis with affinely-parameterized dictionaries and  \ac{SDR} noise. 

The robust counterpart method may be used in other optimization domains with Lie constraints, such as in optimal control input penalties \cite{majumdar2014convex} and maximum controlled invariant set estimation \cite{korda2014convex}.

One future research direction is to investigate other types of exploitable uncertainty structure in dynamical systems problems.
Another avenue is to incorporate warm starts into \ac{SDP} solvers so online system estimates can be updated as more data gets added to $\mathcal{D}$. The bottleneck of all computations developed in this paper is the compuational cost of the SOS tightenings, so finding ways to speed up SOS-based SDP solution is of vital interest. An important research path from a theoretical perspective is finding circumstances in which Assumption A6' can be relaxed while still obeying polynomial approximability Theorem \ref{thm:multipliers_cont_poly}.





%% file: appendix/appendix_polynomial_aux.tex
\section{Polynomial Approximation of the Auxiliary Function}
\label{app:poly_robust_aux}
This appendix uses arguments from \cite{fantuzzi2020bounding} to prove that Problem \eqref{eq:poly_peak_w} may be approximated with $\varepsilon$-accuracy by a polynomial auxiliary function. Assumptions A1-A4 are in place, ensuring that $\Omega = [0, T] \times X \times W$ is compact.

Let $\varepsilon > 0$ be an optimality bound, and let $v(t, x) \in C^1([0, T] \times X)$ be an auxiliary function that satisfies constraints \eqref{eq:poly_peak_w_lie} and \eqref{eq:poly_peak_w_cost} with
\begin{equation}
    \sup_{x \in X_0} v(0, x) \leq d^* + \varepsilon.
\end{equation}

The $i$-th coordinate of dynamics $\dot{x} = F(t, x, w) = f_0(t,x) + \sum_{w=1}^L w_\ell f_\ell(t,x)$ from \eqref{eq:disturbance_affine} is indexed by $F_i(t, x, w)$.

A tolerance $\eta > 0$ may be chosen as (Equation 4.10 of \cite{fantuzzi2020bounding}):
\begin{align}
    \eta < \frac{\varepsilon}{\max\left(2, 2T, 2T \norm{F_1}_{C^0(\Omega)}, \ldots, 2T \norm{F_n}_{C^0(\Omega)} \right)}.
\end{align}

A polynomial approximation of the $C^1$ function $v$ may be performed by Theorem 1.1.2 of \cite{llavona1986approximation} to find a polynomial $w \in \R[t, x]$ such that $\norm{v(t, x) - w(t, x)}_{C^1(\Omega)} < \eta$ uniformly. The perturbed auxiliary function,
\begin{equation}
    V(t, x) = w(t, x) + \varepsilon(1 - t/(2T)),
\end{equation}
satisfies the following strict inequalities from \eqref{eq:poly_peak_w} (equation 4.12 in \cite{fantuzzi2020bounding}),
\begin{subequations}
\label{eq:poly_peak_w_pert}
\begin{align}
    & d^* + (5/2) \varepsilon > V(0, x)  & &  \forall x \in X_0 \\
    & \Lie_{F(t, x, w)} V(t, x) < 0 & & \forall (t, x, w) \in \Omega \label{eq:poly_peak_w_pert_lie}\\
    & V(t, x) < p(x) & & \forall (t, x) \in [0, T] \times X. \label{eq:poly_peak_w_pert_cost}
\end{align}
\end{subequations}

There then exists some finite $d$ such that the polynomial $V(t, x)$ with an optimal solution of (at most) $d^* + (5/2)\varepsilon$ has degree $d$ \cite{fantuzzi2020bounding}.

%% file: appendix/appendix_continuous.tex
\section{Continuity of Multipliers}
\label{app:continuity_robust}

This section will prove the continuous selection portion of Theorem \ref{thm:multipliers_cont_poly}.


\subsection{Set-Valued Preliminaries}
We first review concepts in set-valued analysis.
Given spaces $Y$ and $Z$, a set-valued function $F: Y \rightrightarrows Z$ is a mapping between the power sets $F: 2^Y \rightarrow 2^Z$. A set $E \subset Y$ is inside the domain $\text{Dom}(F)$ if $F(E) \neq \varnothing$. In this section, we will be utilizing point-set maps $(F: Y \rightarrow 2^Z)$.

\begin{defn}[Definition 1.4.2 of \cite{aubin2009set}] \label{def:lsc}The function $F$ is \textbf{lower semicontinuous} at $y \in \text{Dom}(F)$ if, for every sequence $\{y_k\}$ converging to $y$ ($\{y_k\}\rightarrow y$), there exists a converging sequence $\{z_k \in F(y)\}$ converging to an element $z \in F(y)$. The map $F$ is lower semicontinuous if it is lower semicontinuous at each $y \in \text{Dom}(F)$.
\end{defn}

\begin{defn}
    Let $F_0, F_1$ be set-valued maps $Y \rightrightarrows Z$. The containment relation $F_0 \subseteq F_1$ holds if $\forall y \in Y: \ F_0(y) \subseteq F_1(y)$.
\end{defn}

\begin{rmk}
Lower semicontinuity in Definition \ref{def:lsc} is also called \textbf{inner semicontinuity} in \cite{rockafellar2009variational}.
\end{rmk}

\begin{defn}[Definition 1 of \cite{denel1979extensions}]
\label{def:rho_decreasing}
A family of set-valued maps $\{S_\rho: Y \rightrightarrows Z\}_{\rho\geq 0}$ is a $\boldsymbol{\rho}$\textbf{-decreasing family} if $\forall \rho \geq \rho' \geq 0, \ y \in Y: S_{\rho} \subseteq S_{\rho'}$.
\end{defn}

\begin{defn}[Definition 2 of \cite{denel1979extensions}]
\label{def:rho_dense}
    A $\rho$-decreasing family $\{S_\rho\}_{\rho \geq 0}$ is \textbf{dense} if $S_0(y) \subseteq \textrm{Closure}\left(\cup_{\rho>0} S_\rho(y)\right)$ for all $y \in Y$.
\end{defn}

\begin{defn}[Definition 3 of \cite{denel1979extensions}]
    A $\rho$-decreasing family $\{S_\rho\}_{\rho \geq 0}$ is \textbf{pseudo-lower-continuous} at $y$ if for all sequences $\{y_k\}\rightarrow y$, parameters $\rho > \rho' > 0$, and points $z \in S_\rho(y)$, there exists an  $N \in \N$ and a sequence $\{z_k\} \rightarrow z$ such that $\forall k \geq N: z_k \in S_{\rho'}(y_k)$. The family is pseudo-lower-continuous if it is pseudo-lower-continuous at all $y \in Y$.
\end{defn}

\begin{cor}[Excerpt of Remark 5 of \cite{denel1979extensions}] 
\label{cor:rho_0_cont}
    If a $\rho$-decreasing family  $\{S_\rho\}_{\rho \geq 0}$ is pseudo-lower-continuous and dense, then $S_0$ is lower semicontinuous.
\end{cor}

\subsection{Discussion of Assumptions A1'-A5'}

This subsection provides commentary on conditions for which A1'-A5' are valid.

\begin{rmk}
    Continuity of problem data (A3') over the compact $Y$ (A2') implies that all problem entries are finite.
\end{rmk}

\begin{rmk}
    Assumption A4' is a Slater-type condition ensuring feasibility equivalence of the robust inequality and the robust counterpart at each $y \in \textrm{dom}(W)$.
\end{rmk}

Assumption A5' will be satisfied for several common patterns in systems analysis. We list and prove A5' for some of these sets.

\begin{proposition}
    Translates of centrally symmetric sets will satisfy A5'
\end{proposition}
\begin{proof}
    Let us express the set $W$ as $\{w \in \R^L \mid C w + e_+ \in K, \ -C w + e_- \in K\}$. The matrix $A$ for this set is $A = [C; -C]$. For any $\psi \in \textrm{int}(K)$, the choice of $\hat{\zeta} = [\psi; \psi]$ will satisfy $\hat{\zeta} \in \textrm{int}(K)^2$ and $A^T \hat{\zeta} = C \psi - C \psi = 0$.
\end{proof}

The following propositions will rely on Stiemke's alternative (hyperplane separation) generalized to cones:
\begin{lem}[\cite{stiemke1915positive}]
    \label{lem:stiemke}
    Let $K \in \R^z$ be a cone and $\mathcal{A}: \R^L \rightarrow \R^z$ be a linear operator with adjoint $\mathcal{A}^*$. The following statements are strong alternatives:
    \begin{subequations}
    \label{eq:stiemke}
        \begin{align}
            \exists w \in \R^L: & \qquad \mathcal{A}(w) \in K, \ \mathcal{A}(w) \neq 0\\
            \exists \hat{\zeta} \in \textrm{int}(K): & \qquad \mathcal{A}^*(z) = 0. \label{eq:stiemke_a5}
        \end{align}
    \end{subequations}
\end{lem}

\begin{rmk}
    Statement \eqref{eq:stiemke_a5} is a specific case of Assumption A5'.
\end{rmk}

\begin{proposition}
    \ac{PSD} matrices of constant trace will satisfy A5'.
\end{proposition}
\begin{proof}
    Let us express the spectahedron $W$ in terms of matrices $E, \{A_\ell\}_{\ell=1}^L \in \psd^s$ as $\{w \in \R^L \mid E + \sum_{i=1}^L A_\ell w \in \psd^s_+\}$. If all matrices $E + \sum_{i=1}^L A_\ell w$ have a trace that is constant in $w$, then $\textrm{Tr}(E + \sum_{i=1}^L A_\ell w) = \textrm{Tr}(E)$ and $\textrm{Tr}(\sum_{i=1}^L A_\ell w) = 0$. Given that the only \ac{PSD} matrix with trace $0$ is the zero matrix, it holds that there does not exists a $w \in \R^L$ such that $\sum_{i=1}^L A_\ell w \in \psd_+^s$ and $\sum_{i=1}^L A_\ell w \neq 0$. A5' is therefore satisfied by Stiemke's alternative (Lemma \ref{lem:stiemke}).
\end{proof}

\begin{proposition}
\label{eq:polytope_slater}
    Compact and nonempty polytopes $W = \{w \in \R^L \in \Gamma w \leq h\}$ with dimensions $\Gamma \in \R^{m \times n}, h \in \R^n$ will satisfy A5'.
\end{proposition}
\begin{proof}

This proof will proceed by contradiction. Assume there exists an $w_0 \in \R^L$ such that $\Gamma w_0 \leq 0$, $w_0 \neq 0$. 
For any $\tau \geq 0$ and feasible $w \in W$, the quantity $w + \tau w_0$ would then be feasible with $\Gamma (w + t w_0) = \Gamma w + t (\Gamma w_0) = \Gamma w + t 0 \leq b$. However, the quantity $w + \tau w_0$ with $\tau \geq 0$ forms a ray in $\R^L$. This forms a contradiction: it is not possible for the compact $W$ to include all points on an unbounded ray. It is therefore not possible to form a $w_0\neq 0$ with $\Gamma w_0 \leq 0$, which implies the existence of a $\hat{\zeta} \in \R^m_{>0}$ with $\Gamma^T\hat{\zeta}=0$ by Stiemke's alternative (Lemma \ref{lem:stiemke}).


    
\end{proof}


\subsection{Lower Semicontinuity of Strict Robust Counterpart Multipliers with Constant Parameters}
\label{sec:robust_continuity}
This subsection will analyze continuity properties of the strict semi-infinite inequality \eqref{eq:lin_robust_strict}.



We define a $\rho$-indexed family of set-valued maps $S_\rho: Y \rightrightarrows Z$ as the $\rho$-modified solution map to \eqref{eq:lin_robust}:
\begin{subequations}
\label{eq:sol_zeta}
\begin{align}
    S_{\rho > 0}(y) &= \left\{(\zeta, \beta) \in Z: \begin{array}{r} e^T \zeta + a_0^T\beta +\rho \leq b_0 \\ G^T \zeta = 0 \\ A^T \zeta + a_\bullet^T \beta = b_\bullet \end{array}\right\}, \\
    S_{0}(y) &= \left\{(\zeta, \beta) \in Z: \begin{array}{r} e^T \zeta + a_0^T\beta  < b_0 \\ G^T \zeta = 0 \\ A^T \zeta + a_\bullet^T \beta = b_\bullet \end{array}\right\} \label{eq:sol_zeta_0} \\
    S_{0}^\beta (y) &= \left\{\beta \in \R^L: \ \exists \zeta \mid (\zeta, \beta) \in S_0(\zeta, \beta)\right\}.\label{eq:sol_beta_0}
\end{align}
\end{subequations}

The set $S_0^\beta(y)$ from \eqref{eq:sol_beta_0} is the $\beta$-projection of the solution map \eqref{eq:sol_zeta_0}, and is equivalently the feasible set of solutions $\beta$ for \eqref{eq:lin_robust_strict_param}.

The semi-infinite strict program \eqref{eq:lin_robust_strict} has a solution at the parameter $y$ if $S_0(y) \neq \varnothing$.

\begin{lem}
\label{lem:s_rho_decrease}
The family $\{S_\rho\}_{\rho \geq 0}$ from \eqref{eq:sol_zeta} is a $\rho$-decreasing family (Def. \ref{def:rho_decreasing}).
\end{lem}
\begin{proof}
The tolerance $\rho$ only appears in the linear inequality $e^T \zeta + a_0^T\beta +\rho \leq b_0$. The maps therefore satisfy $\forall \rho > \rho' \geq 0: \ S_\rho \subseteq S_{\rho'}$.
\end{proof}

\label{eq:robust_continuity_slater}

Assumption A5' will allow us to use a strong strict Slater condition \cite{daniilidis2013lower} to prove lower semicontinuity of $S_\rho$.



An element $(\zeta, \beta) \in Z$ is a \textit{Slater point} of $S_\rho(y)$ if $(\zeta, \beta) \in S_\rho(y)$ and $\zeta \in \textrm{int}(K)$.

For every $(\zeta, \beta) \in S_\rho(y)$ with $\rho > 0$, we can construct a $\hat{\zeta} \in \textrm{int}(K)$ such that $(\zeta + \hat{\zeta}, \beta) \in S_\rho(y)$ and $(\zeta + \hat{\zeta}, \beta)$ is a Slater point of $S_{\rho/2}(y)$.

By assumption A5', there exists a $\hat{\zeta} \in \textrm{int}(K)$ such that $A(y)^T \hat{\zeta} = 0$ and $G(y)^T \hat{\zeta} = 0$. Let us choose $\hat{\zeta}$ such that $e^T \hat{\zeta} = \rho/2$. Then we have
\begin{subequations}
\begin{align}
    e^T \zeta + \alpha_0 \beta + 2(\rho/2) & \leq b_0 \\
    e^T \zeta + (e^T \hat{\zeta}) + \alpha_0 \beta + \rho/2 & \leq b_0
\end{align}
\end{subequations}
and that  $G^T (\zeta + \hat{\zeta}) = G^T \zeta,$ $A^T (\zeta + \hat{\zeta}) + a^T_\bullet \beta = A^T \zeta + a^T_\bullet \beta$, 
thus ensuring that $(\zeta+\hat{\zeta}, \beta) \in S_{\rho/2}(y)$. 
The (strong) Slater point characterization \cite{daniilidis2013lower} states that the existence of Slater points are necessary and sufficient to prove lower-semicontinuity in case of $y$-perturbations on the left-hand-side (A4). Because there exists a Slater point $(\zeta+\hat{\zeta}, \beta) \in S_{\rho/2}(y)$ for each $(\zeta, \beta) \in S_{\rho}(y)$, that  $S_\rho$ is a $\rho$-decreasing family ($S_\rho \subseteq S_{\rho/2}$), $S_\rho$ has closed convex images for $\rho>0$, and $S_\rho$ sends a compact set $Y$ (A2') to a Banach space $(\textrm{aff}(K^*) \times \R^r)$, it holds that $S_\rho$ is lower-semicontinuous on its domain.

 \subsection{Continuity of Multipliers}

 
 \label{sec:continuity_lie_robust}



We now review a condition for a continuous selection:
\begin{defn}
    Let $F: Y \rightrightarrows Z$ be a set-valued map. The function $\sigma: Y \rightarrow Z$ is a \textbf{selection} for $F$ if $\forall y \in Y: \ \sigma(y) \in F(y)$.
\end{defn}

\begin{thm}[Michael's Theorem, Thm. 9.1.2 of \cite{aubin2009set}]
\label{thm:michaels}
    Let $Y$ be a compact metric space and $Z$ be a Banach Space. If $F: Y \rightrightarrows Z$ has closed convex images for each $y \in Y$, then there exists a continuous selection $\sigma$ for $F$.
\end{thm}

Note that Michael's Theorem does not require that the images of $F$ in $Z$ should be compact.
Michael's theorem simply requires closed convex images.

\begin{prop}\label{prop:tau_pos}
Under assumptions A1'-A5' and with lower-semicontinuity of $S_\rho$, the following quantity $\tau$ is positive:
\begin{align}
    \tau = & \inf_{y, w, \beta} b_0(y) +  b_\bullet(y) w - \beta^T(a_0(y) +  a_\bullet(y))  w \\
    & y \in Y, \ w \in W(y), \ \beta \in S^\beta_0(y) \label{eq:tau_pos}
\end{align}
\end{prop}

\begin{rmk}
    Removing the compactness requirement A2' could cause $\tau$ to be zero.
\end{rmk}


We now prove Theorem \ref{thm:multipliers_cont_poly}:
\begin{proof}
     Given a lower bound $\tau > 0$ from \eqref{eq:tau_pos}, all mappings $\{S_{\rho}\}_{\rho \in [0, \tau]}$ are equal to each other. It therefore holds that both $S_\tau$ and $S_0$ are closed. $S_0$ satisfies all requirements of Michael's Theorem \ref{thm:michaels} and therefore has a continuous selection for the Lie multipliers.  
The below minimal map is one such continuous selection  (Prop. 9.3.2 in \cite{aubin2009set}):
    \begin{equation}
        \label{eq:min_norm_map}
   m(S_0(y)) \doteq \left \{ (\zeta, \beta) \in S_0(y) \colon
\norm{\zeta} + \norm{\beta} = \min_{(\zeta, \beta)  \in  S_0(y)} \norm{\zeta} + \norm{\beta}  \right \}.
 \end{equation}
\end{proof}

%% file: appendix/appendix_polynomial_mult_AG.tex
\section{Polynomial Approximation of Multipliers}
\label{app:poly_robust_mult}

Let $(\zeta^c(y), \beta^c(y))$ be a continuous selection of multipliers of $S_0(y)$ \eqref{eq:sol_zeta_0} (guaranteed to exist by Michael's Theorem \ref{thm:michaels}). This section will prove that there exists a polynomial choice $(\zeta^p(y), \beta^p(y))$  that is also a continuous selection for $S_0$.

\subsection{Choice of Free Variable}

We begin by performing a Stone-Weierstrass approximation of $\beta^c(y)$ by a polynomial $\beta^p(y)$:
\begin{equation}
    \sup_{y \in Y} \ \norm{\beta^c(y)-\beta^p(y)}_\infty \leq \epsilon_\beta.
\end{equation}

Letting $r_\beta$ be the residual $\beta^c = \beta^p + r_\beta, \ \norm{r_\beta}_{C^0(Y)} \leq \epsilon_\beta$, we can express the robust counterpart expression \eqref{eq:lin_robust_strict_param} as
\begin{subequations}
    \begin{align}
        &e^T \zeta + a_0^T(\beta^p + r_\beta)  < b_0  \label{eq:robust_betap}\\ 
        &G^T \zeta = 0 \\ 
        &A^T \zeta = b_\bullet - a_\bullet^T (\beta^p + r_\beta)
    \end{align}
\end{subequations} 

Given that $a_0(y)^T r_\beta \leq \epsilon_\beta \norm{a_0(y)}_1$ and letting $a_0^* = \sup_{y \in Y} \  \norm{a_0(y)}_1$, a choice of $\epsilon_\beta <  \tau / (4 a_0^*)$ will preserve the strict inequality for \eqref{eq:robust_betap} to produce 
\begin{equation}
        e^T \zeta + a_0^T\beta^p +  \epsilon_\beta a_0^* < b_0  \label{eq:robust_betap_2}\\ 
\end{equation}
\subsection{Continuous Parameterization}


Let us define the matrix $\Phi$ as
\begin{equation}
    \Phi = [G; A]^T \\
\end{equation}
By assumption in Theorem \ref{thm:multipliers_cont_poly}, the matrix $\Phi$ is constant in $y$. 
Define $H$ as a constant matrix whose columns span the nullspace of $\Phi$, in which $N$ is the nullity of $\Phi$. The following least-squares solutions can be taken (ignoring the conic constraint $\zeta \in K^*$):
\begin{subequations}
\begin{align}
    \theta(y) &= \Phi^+ [\mathbf{0}; b_\bullet(y) - a_\bullet^T(y) \beta^p(y)] \\
    \phi(y) &= \Phi^+ [\mathbf{0};  - a_\bullet^T(y) r_\beta(y)].
\end{align}
\end{subequations}

The vectors $\theta(y)$ and $\phi(y)$  are continuous functions of $y$ given that $a_\bullet$ and $b_\bullet$ are continuous (A3'), $\Phi$ is constant,  $\beta^p$ is polynomial, and $r_\beta$ is continuous 

The multipler $\zeta^c(y)$ can be presented using a continuous function $\psi(y): Y \rightarrow  \R^N$ as 
\begin{equation}
    \zeta^c(y) = \theta(y) + H \psi(y) + \phi(y), \label{eq:zeta_decompose}
\end{equation}

We will partition $(\theta, H, \phi)$ according to the resident cones $K_s$ by
\begin{align}
    \zeta_s &= \theta_s + H_s \psi + \phi_s& & \forall s=1..N_s.
\end{align}


\subsection{Polynomial Approximation}
We will use the Stone-Weierstrass theorem to approximate the functions $\psi_c: Y \rightarrow \R^N$, by a polynomial vector $\psi_p \in \R[y]^N$ in the compact space $Y$ up to a tolerance $\epsilon_\psi>0$ with
\begin{align}
\label{eq:robust_sw}
    \sup_{y \in Y} &\ \norm{\psi_c(y) - \psi_p(y)}_\infty \leq \epsilon_\psi.
\end{align}
A similar approximation will take place for $\phi$ with 
\begin{align}
\label{eq:robust_sw_phi}
    \sup_{y \in Y} &\ \norm{\phi_c(y) - \phi_p(y)}_\infty \leq \epsilon_\phi.
\end{align}
In order to pose a valid approximation $(\zeta^p, \beta^p)$, we need to use a notion of centers of cones. We will choose the incenter:

\begin{defn}[Def. 2.1 of \cite{henrion2010properties}] 
Let $(X, \norm{\cdot})$ be a reflexive Banach space with distance $\text{dist}(x_1, x_2) = \norm{x_1-x_2}$, and let $S_X$ be the unit sphere in $X$.
Given a cone $K \subset X$, let $K \cap S_X$ be the set of unit-norm elements of the cone $K$. The \textbf{incenter} of $K$ is the unique solution to
\begin{align}
    \varsigma(K) = \sup_{x \in K \cap S_X} \textrm{dist}(x; \partial K).
\end{align}
\end{defn}

\begin{rmk}
The following equation lists common cones and their incenters \cite{henrion2010inradius}:
\begin{subequations}
\label{eq:incenters}
    \begin{align}
        \R_{\geq 0}: \ & 1 & Q^n: \ & (\0_n; 1)         & \psd_+^n: \ & I_n / \sqrt{n}.  
    \end{align}
    \end{subequations}
\end{rmk}

For a given cone $K_s$, we define $c_s$ as the incenter of $K_s$. In the semidefinite case, the incenter will be appropriately vectorized following the vectorial convention of cone containment $K_s$.

Our approximation $(\zeta^p, \beta^p)$ will be defined using tolerances $\delta_s > 0$ for $s=1..N_s$:
\begin{align}
\label{eq:robust_mult_poly}
    \zeta^p &= \theta_s + H_s \psi^p + \phi^c_s + \delta_s c_s & & \forall s=1..N_s 
\end{align}

The tolerance terms $\delta_s c_s$ will encourage conic containment in $K^*$. 

The approximator $\zeta^p$ is related to $\zeta^c$ by
    \begin{align}
        \zeta^p_s &= \zeta^c_s + \delta_s c_s + H_s(\psi^p - \psi^c) + (\phi^p - \phi^c) \label{eq:zeta_p_rel}.
    \end{align}

The term in \eqref{eq:zeta_p_rel} dominates the worst-case bound
\begin{align}
    \zeta^c_s + c_s\delta_s + H_s(\psi^p - \psi^c) + (\phi^p - \phi^c)\geq_{K^*_s} \zeta^c_s + c^*_s(\delta_s - \norm{H_s}_\infty \epsilon_\psi - n_s \epsilon_\phi)
    \label{eq:zeta_p_bound}
\end{align}
using the Stone-Weierstrass approximation \eqref{eq:robust_sw}.

A sufficient condition for $\zeta^p_s \in K^*_s$ through \eqref{eq:zeta_p_bound} is that 
\begin{equation}
\delta_s \geq \norm{H_s}_\infty \epsilon_\psi + n_s \epsilon_\phi. 
   \label{eq:robust_cond_sw}    
\end{equation}

We now move to the strict inequality constraint in \eqref{eq:robust_betap}
\begin{equation}
\label{eq:sol_zeta_strict_ineq}
    e^T \zeta + a_0^T\beta^p +  \epsilon_\beta a_0^* < b_0 .
\end{equation}

Substitution of $\zeta^p$ from \eqref{eq:robust_mult_poly} into \eqref{eq:sol_zeta_strict_ineq} leads to
\begin{align}\label{eq:sol_zeta_strict_ineq_subs} 
   \textstyle \sum_{s=1}^{N_s} e^T_s (\zeta^c_s + c_s \delta_s + H_s(\psi^p - \psi^c) + (\phi^p - \phi^c)) + a_0^T(\beta^p) + \epsilon_\beta a_0^* < b_0.
\end{align}

Recalling that each cone $K_s$ is a subset of the finite-dimensional $\R^{n_s}$, the left-hand term of \eqref{eq:sol_zeta_strict_ineq_subs} is upper-bounded using \eqref{eq:robust_sw} and \eqref{eq:robust_sw_phi} by
\begin{subequations}
    \label{eq:sol_zeta_strict_ineq_subs_bound} 
\begin{align}
    &\textstyle\sum_{s=1}^{N_s} (e^T_s \zeta^c_s + e^T_s c_s \delta_s + \norm{\diag{e_s} H_s}_\infty \epsilon_\psi n_s + \epsilon_\phi n_s) + a_0^T(\beta^p) + \epsilon_\beta a_0^*\\
    =&\textstyle\left[ \sum_{s=1}^{N_s}e^T_s \zeta^c_s + a_0^T\beta^p \right] + \left[\sum_{s=1}^{N_s}(\delta_s e^T_sc_s  + \norm{\diag{e_s} H_s}_\infty n_s \epsilon_\psi + \epsilon_\phi n_s) +  \epsilon_\beta a_0^* \right].
\end{align}
\end{subequations}

Define $Q^*$ as the finite and positive value
\begin{align}
\label{lem:max_value_q}
    Q^* = \min_{y \in Y} b_0 - \left[ \textstyle \sum_{s=1}^{N_s}e^T_s \zeta^c_s + a_0^T\beta^p \right] > 0.
\end{align}
The minimum of \eqref{lem:max_value_q} is attained because all functions $(b_0, e, a_0, \zeta, \beta)$ are continuous in the compact region $Y$. 

Successful polynomial-based approximation with $(\zeta^p(y), \beta^p(y)) \in S_0(y)$ will occur if $(\{\delta_s\}, \epsilon)$ are chosen with
\begin{subequations}
\label{eq:poly_suff_both}
\begin{align}
& \epsilon_\beta, \ \epsilon_\psi, \ \epsilon_\phi > 0 \\
& \epsilon_\beta <  \tau / (4 a_0^*) \\
    &\forall s = 1..N_s: \qquad \delta_s\geq \norm{H_s}_\infty \epsilon_\psi + n_s \epsilon_\phi \\
    &\textstyle\sum_{s=1}^{N_s}(\delta_s e^T_s c_s + \norm{\diag{e_s} H_s}_\infty n_s \epsilon_\psi + n_s \epsilon_\phi ) \label{eq:poly_suff_both_delta} +  \epsilon_\beta a_0^* \epsilon <  Q^*.
\end{align}
\end{subequations}

Shrinking the tolerances $(\{\delta_s\}, \epsilon_\beta, \epsilon_\phi, \epsilon_\psi)$ towards zero will result in approximations of increasing quality. 
This approximation quality is directly relevant towards establishing convergent bounds in Lie problems, such as in the suboptimal peak estimation task discussed in Appendix \ref{app:poly_robust_aux}.

%% file: appendix/appendix_robust_occ_duality.tex
\section{Robust Duality and Recovery}
\label{app:robust_duality}

This appendix dualizes programs formed by robust Lie constraints \eqref{eq:lie_robust_counterpart} and forms an interpretation based on occupation measures. It also reviews a technique from \cite{majumdar2014convex, korda2015controller} to extract approximate polynomial control laws from moment-\ac{SOS} solutions.

\subsection{Measure Theory}

The dual \acp{LP} of the function programs described in this paper are \acp{LP} in measures. We briefly review concepts in measure theory in this subsection. Refer to \cite{tao2011introduction} for a complete reference, and to \cite{ pivato2003analysis} for a visual introduction to measure theory.

\subsubsection{Nonnegative Measures}
Let $S$ be a Banach space, and let $\textrm{Set}(S) = 2^S$ be the set of sets (power set) of a space $S \in \R^n$. A $\sigma$-algebra $\Sigma$ over $S$ is a subset of $\textrm{Set}(S)$ such that $\Sigma$ contains $S$ and is closed under countable unions and complements. Examples of $\sigma$-algebras over the real line $\R$ include the countable unions of the set of intervals $[a, b]$ with $a,b \in \R$ and when $a, b \in \Z$.


A nonnegative Borel measure $\mu: \Sigma \rightarrow \R_+$ is a function that assigns a size (measure) to each set $A \subset S$ in a $\sigma$-algebra $\Sigma$ under the following rules:
    \begin{enumerate}
        \item $\mu(A) \geq 0 \qquad \forall A \in \Sigma$
        \item $\mu(\varnothing) = 0$ \label{enum:meas_2}
        \item $\mu(\bigcup_{k=1}^\infty A_k) = \sum_{k=1}^\infty \mu(A_k)\qquad $ sets $A_k$ are disjoint. \label{enum:meas_3}
    \end{enumerate}
The set of all nonnegative Borel measures over $S$ is $\Mp{S}$. The mass of a nonnegative measure $\mu \in \Mp{S}$ is $\mu(S)$, and $\mu$ is a probability measure if $\mu(S)=1$.  The Dirac Delta $\delta_{s'}$ at a base point $s' \in S$ is the unique probability measure satisfying $\mu(A) = 1$ whenever $s' \in A$ and $\mu(A) = 0$ otherwise. The support of a measure $\mu$ is the locus of all points $s$ where every open neighborhood $Nb \ni s$ satisfies $\mu(Nb)>0$.

\subsubsection{Absolute Continuity and Domination}
\label{sec:prelim_abscont_dom}

Let $\mu, \nu \in \Mp{S}$ be nonnegative Borel measures. The measure $\nu$ is absolutely continuous to $\mu$ ($\nu \ll \mu$) if, for every $A \in \Sigma$, $\mu(A)=0$ implies that $\nu(A) = 0$. Equivalently, there exists a unique  nonnegative and measurable (density) function $\rho(s)$ such that $\inp{f(s)}{\mu(s)} = \inp{f(s) \rho(s)}{\nu(s)}$ for all $f \in C(S)$. The function $\rho$ is the Radon-Nikod\'ym derivative $\rho = d \nu / d \mu$.

The measure $\nu$ is dominated by $\mu$ if $\nu(A) \leq \mu(A)$ for all subsets $A \subset S$ (elements $A$ in a $\sigma$-algebra).
There exists a unique nonnegative slack measure $\hat{\nu} \in \Mp{S}$ such that $\nu(A) + \hat{\nu}(A) = \mu(A), \ \forall A \subset S$, which may be equivalently written as $\nu + \hat{\nu} = \mu$. Domination ($\nu \leq \mu$) is a stronger condition than absolute continuity ($\nu \ll \mu$).

A pair of measures $\omega, \nu \in \Mp{S}$ are orthogonal ($\omega \perp \nu$) if $\forall A \in \supp{\omega}: \nu(A) =0$ and $\forall B \in \supp{\nu}: \mu(B) =0$. While the sum $\nu + \omega$ dominates $\nu$ and $\omega$ individually, it does not hold that every $\mu$ with $\mu \geq \nu$ and $\mu = \nu + \hat{\nu}$ produces an orthogonal pair $\nu \perp \hat{\nu}$. 


\subsubsection{Signed Measures}
A signed measure is a function $\mu: \Sigma \rightarrow \R$ that only satisfies conditions \ref{enum:meas_2} and \ref{enum:meas_3} of a nonnegative measure (sets may have negative measure). The set of signed measures over a space $S$ is $\mathcal{M}(S)$. 
The Hahn-Jordan decomposition is a unique method to split a signed measure $\mu \in \mathcal{M}$ into the difference of two orthogonal nonnegative measures $\mu = \mu^+ - \mu^-$. 

The \ac{TV} norm of a signed measure $\mu \in \mathcal{M}(S)$ is
\begin{subequations}
\label{eq:tv_norm}
\begin{align}
    \norm{\mu}_{TV} &= \sup_{v \in C(S)} \inp{v}{\mu}: \ -1 \leq v(s) \leq 1 \ \forall s \in S \\
    &= \inf_{\mu^+, \ \mu^- \in \Mp{S}} \inp{1}{\mu^+} + \inp{1}{\mu^-}: \ \mu^+ - \mu^- = \mu.
\end{align} 
\end{subequations}

\subsubsection{Pairings and Moments}
The sets $\mathcal{M}(S)$ and $C(S)$ are topological duals when $S$ is compact under the  pairing $\inp{f}{\mu} = \int_S f d\mu \ \forall f \in C(S), \ \mu \in \mathcal{M}(S)$. This pairing forms a  duality pairing when restricted to the nonnegative subcones $\Mp{S}$ and $C_+(S)$, in which $\inp{f}{\mu} \geq 0 \ \forall f \in C(S), \mu \in \Mp{S}$. The pairing between a $f \in C(S)$ and the Dirac Delta $\delta_{s'}$ is $\inp{f}{\delta_{s'}} = f(s')$. Every linear operator $\mathcal{P}: S \rightarrow Y$ possesses an adjoint $\mathcal{P}^\dagger$ such that $\forall f \in C(Y), \ \mu \in \mathcal{S},  \ \inp{f(\mathcal{P}(s))}{\mu(y)} = \inp{f(y)}{(\mathcal{P}^\dagger \mu)(s)}$.

Let $S$ be an $n$-dimensional space and let $\alpha \in \N^n$ be a multi-index. The $\alpha$-moment of a measure $\mu \in \Mp{S}$ is $\bbm_\alpha = \inp{s^\alpha}{\mu}$. The mass is the $0$-moment $\bbm_0 = \inp{1}{\mu}$. A moment sequence is an infinite collection of moments $\{\bbm_\alpha\}_{\alpha \in \N^n}$. The moment matrix of a sequence $\bbm$ is the infinite-dimensional matrix indexed by multi-indices 
$\alpha, \beta \in \N^n$ with $(\M[\bbm])_{\alpha+\beta} = \bbm_{\alpha+\beta}$. The degree-$d$ truncation $\M_d[\bbm]$ is the $\binom{n+d}{d}$-size symmetric matrix involving moments up to degree $2d$. A sufficient conditions for a sequence of numbers (pseudo-moments) $\{\tilde{\bbm}\}_{\alpha \in \N^n, \ \abs{\alpha} \leq 2d}$ to have at least one measure $\mu \in \Mp{\R^n}$ satisfying $\inp{s^\alpha}{\mu} = \tilde{\bbm}_\alpha$ is that $\M_d[\bbm]$ is \ac{PSD} \cite{lasserre2009moments}. This possibly non-unique $\mu$ is called a representing measure of $\tilde{\bbm}.$

Let $\K = \{s \mid \forall j=1..N_c: g_j(s) \geq 0\}$ be a \ac{BSA} set with scalar polynomial constraints $g_j(s) \geq 0$.
The localizing matrix associated with a polynomial $g(s) = \sum_{\gamma \in \N^n} g_\gamma s^\gamma$ and a moment sequence $\bbm$ is $(\M[g \bbm])_{\alpha \beta} = \sum_{\gamma \in \N^n} g_\gamma \bbm_{\alpha+\beta+\gamma}$. If $\K$ is Archimedean, a necessary and sufficient condition for the pseudo-moments $\tilde{\bbm}$ to have a representing measure supported on $\K$ is that $\M_d[\tilde{\bbm}]$ and all $\M_{d - \floor{\deg g_j/2}}[\tilde{\bbm}]$ are \ac{PSD} for every $j=1..N_c$.


\subsubsection{Occupation Measures}

Let $[0, T]$ be a time range and let $s: [0, T] \rightarrow S$ be a curve on $S$. The \textit{occupation measure} of $t \mapsto (t, s(t))$ in the time range $[0, T]$ is the unique measure $\mu$ such that $\forall v \in C([0, T] \times S): \ \inp{v}{\mu} = \int_{0}^T v(t, s(t)) dt$.  

The dual of the peak estimation problem \eqref{eq:poly_peak_w} involving an initial measure $\mu_0$, a peak measure $\mu_p$, and a (relaxed) occupation measure $\mu$ is \cite{miller2021uncertain, lewis1980relaxation}
\begin{subequations}
\label{eq:peak_meas}
    \begin{align}
        p^* =&\quad  \sup \ \inp{p(x)}{\mu_p} \label{eq:peak_meas_obj}\\
        & \inp{v(t, x)}{\mu_p(t, x)} = \inp{v(0, x)}{\mu_0(x)} + \inp{\Lie_f v(t, x, w)}{\mu(t, x, w)} \label{eq:peak_meas_liou}\\        
        & \inp{1}{\mu_0} = 1\label{eq:eq_meas_prob} \\
        & \mu_0 \in \Mp{X_0}, \ \mu_p \in \Mp{[0, T] \times X} \\
        & \mu \in \Mp{[0, T] \times X \times W}.\label{eq:eq_meas_occ}
    \end{align}
\end{subequations}
Constraint \eqref{eq:peak_meas_liou} is a Liouville equation that ensures $\mu_0$ and $\mu_p$ are connected together by the dynamical system $f$, whose trajectories are represented by $\mu$. Constraint \eqref{eq:eq_meas_prob} ensures that the initial measure $\mu_0$ is a probability measure. Letting $t^*$ be a stopping time, $x_0 \in X_0$ be an initial condition, and $w(\cdot)$ be an admissible input process (constrained to lie in $W$ with no assumption of continuity), the measures $\mu_0 = \delta_{x = x_0}, \ \mu_p = \delta_{t = t^*, \ x=(t^* \mid x_0, w(\cdot))},$ and $\mu$ as the occupation measure $t \mapsto (t, x=(t^* \mid x_0, w(\cdot)), w(t))$ in times $[0, t^*]$ are feasible solutions to the constraints of \eqref{eq:peak_meas}.
Under assumptions A1-A6, program \eqref{eq:peak_meas} achieves strong duality with \eqref{eq:poly_peak_w}. In this instance, $\mu$ from \eqref{eq:eq_meas_occ} is a controlled (Young) measure \cite{young1942generalized}.

\subsection{Duality}

To simplify explanations, we will consider a polytope-constrained peak estimation problem from \eqref{eq:poly_peak_w_robust} with $G=0$ and an uncertainty set $W = \{w \in \R^L \mid e - A w \geq 0\}$. The Lie-robustified peak estimation \ac{LP} under polytopic uncertainty considered in this appendix is
\begin{subequations}
\label{eq:poly_peak_w_robust_sp}
\begin{align}
    d^* = & \ \inf_{v(t,x), \gamma, \zeta} \quad \gamma & \\
    & \gamma \geq v(0, x)  & &  \forall x \in X_0 \\
& \Lie_{f_0} v(t, x) + e^T \zeta(t, x) \leq 0 & & \forall (t, x) \in [0, T] \times X  \label{eq:poly_peak_w_robust_sp_lie}\\
    &  -(A^T)_{\ell}\zeta(t,x) + f_\ell \cdot \nabla_x v(t,x) =0 & & \forall \ell=1..L \label{eq:poly_peak_w_robust_sp_dual} \\ 
    & v(t, x) \geq p(x) & & \forall (t, x) \in [0, T] \times X  \\
    & v(t, x)\in C^1([0, T]\times X) \\
    & \zeta_j(t, x) \in C_+([0, T] \times X) & & \forall j = 1..m.
\end{align}
\end{subequations}

We will derive a dual to \eqref{eq:poly_peak_w_robust_sp}.
Define the following measures as multipliers to constraints in \eqref{eq:poly_peak_w_robust_sp}:
\begin{subequations}
\label{eq:robust_measures}
    \begin{align}
        &\textrm{Initial}  & \mu_0 &\in \Mp{X_0} \\
        &\textrm{Occupation} & \mu &\in \Mp{[0, T] \times X} \\
        &\textrm{Peak}  & \mu_p &\in \Mp{[0, T] \times X} \\
        &\textrm{Controlled}  & \nu &\in \mathcal{M}([0, T] \times X)  & & \forall \ell=1..L \\
        &\textrm{Constraint-Slack}  & \hat{\mu}_j &\in \Mp{[0, T] \times X}  & & \forall j=1..m.
    \end{align}
\end{subequations}

The Lagrangian $\scL$ associated with \eqref{eq:poly_peak_w_robust_sp} is
\begin{align}
    \scL =& \gamma + \inp{v(0, x)-\gamma}{\mu_0} + \inp{\Lie_{f_0} v(t, x) + e^T \zeta(t,x)}{\mu} + \inp{-v(t, x) + p(x)}{\mu_p} \\
    &+ \textstyle \sum_{j=1}^m \inp{-\zeta_j}{\hat{\mu}_j} + \sum_{\ell=1}^L \inp{f_\ell \cdot \nabla_x v(t, x) - (A^T)_\ell \zeta(t, x)}{\nu_\ell}\nonumber \\
    =& \textstyle \gamma(1 - \inp{1}{\mu_0}) + \inp{v(t, x)}{\delta_0 \otimes \mu_0 \Lie_{f_0}^\dagger\mu + \sum_{\ell=1}^L (f_\ell \cdot \nabla_x)^\dagger \nu_\ell - \mu_p} \\
    &+\textstyle \sum_{j=1}^m \inp{\zeta_j(t, x)}{ e_j \mu-(\sum_{\ell=1}^L A_{j\ell}\nu_j) - \hat{\mu}_j } + \inp{p(x)}{\mu_p}.
\end{align}

The dual measure \ac{LP} of \eqref{eq:poly_peak_w_robust_sp} is
\begin{subequations}
\label{eq:robust_meas}
    \begin{align}
        p^* =& \quad \sup_{\textrm{\eqref{eq:robust_measures}}} \  \inf_{\gamma, v, \zeta} \scL \\
        =&\quad  \sup \ \inp{p(x)}{\mu_p} \label{eq:robust_meas_obj}\\
        & \mu_p = \delta_0 \otimes \mu_0 + \Lie_{f_0}^\dagger\mu + \textstyle \sum_{\ell=1}^L (f_\ell \cdot \nabla_x)^\dagger \nu_\ell \label{eq:robust_meas_liou}\\
        &  e_j \mu= (\textstyle \sum_{\ell=1}^L A_{j\ell}\nu_j) + \hat{\mu}_j  & & \forall j = 1..m \label{eq:robust_meas_dom}\\
        & \inp{1}{\mu_0} = 1\label{eq:robust_meas_prob} \\
        & \textrm{Measures from \eqref{eq:robust_measures}}.           
    \end{align}
\end{subequations}

\begin{rmk}
Program \eqref{eq:robust_meas} should be compared against the standard peak estimation program \eqref{eq:peak_meas}. Constraint \eqref{eq:robust_meas_liou} is a robustified Liouville equation. Constraint \eqref{eq:robust_meas_dom} is a sequence of domination conditions, as detailed in Section \ref{sec:prelim_abscont_dom}. Constraint \eqref{eq:robust_meas_prob} enforces that $\mu_0$ is a probability measure.
\end{rmk}

\begin{lem}
\label{lem:robust_meas_upper}
    The measure program \ref{eq:robust_meas} upper-bounds on \eqref{eq:peak_traj_poly} with $p^* \geq P^*$.
\end{lem}
\begin{proof}
    Let $t^* \in (0, T]$ be a stopping time of a trajectory of \eqref{eq:disturbance_affine} with applied control input $w(t)$ starting from an initial condition $x_0 \in X_0$. Measures from \eqref{eq:robust_meas} may be constructed from the trajectory $x(t \mid x_0, w(\cdot))$. 

The probability measures are $\mu_0= \delta_{x=x_0}$ and $
    \mu_p = \delta_{t=t^*, x=x(t^* \mid x_0, w(\cdot))}$. Relaxed occupation measures may be chosen as the occupation measures of the following evaluation maps in the times $t \in [0, t^*]$:
    \begin{subequations}
\begin{align}
    \mu: \quad & t \mapsto (t, x(t \mid x_0, w(\cdot)) \\
    \nu_\ell:\quad  & t \mapsto (t, w_\ell(t) x(t \mid x_0, w(\cdot)) & & \forall \ell=1..L \\
    \hat{\mu}_j:\quad  & t \mapsto (t, (e_j - A_j w(t)) x(t \mid x_0, w(\cdot)) & & \forall j=1..m.
\end{align}
    \end{subequations}
Every trajectory $x(t \mid x_0, w(\cdot))$ has a feasible measure representation, proving the upper-bounding theorem.
\end{proof}

\begin{thm}
\label{thm:robust_upper_bound}
Strong duality holds between  \eqref{eq:robust_meas} with \eqref{eq:poly_peak_w_robust_sp} $d^*=p^* =P^*$ under assumptions A1-A6.
\end{thm}
\begin{proof}
The bound $d^* \geq p^*$ holds by weak duality \cite{boyd2004convex}. Lemma 
\ref{lem:robust_meas_upper} proves that $p^* \geq P^*$, together forming the chain $d^* \geq p^* \geq P^*$. Corollary \ref{cor:peak} proves that the optimal value $d^*$ from the robust \eqref{eq:poly_peak_w_robust} equals the optimal value of the non-robust \eqref{eq:poly_peak_w}, which is in turn equal to $P^*$ from \eqref{eq:peak_traj_poly} by Theorem 2.1 of \cite{lewis1980relaxation}. Since $d^*$ and $P^*$ are equal, it holds that the sandwiched $p^*$ satisfies $d^*=p^* =P^*$.
\end{proof}

\begin{lem}
\label{lem:robust_mass_bound}
Under A1-A6 and the further assumption that the polytope $W$ is compact all of the nonnegative measures in \eqref{eq:robust_meas} are bounded.
\end{lem}
\begin{proof}
Boundedness of a nonnegative measure will be demonstrated by showing that the measure has finite mass and it is supported on a compact set. Assumptions A1-A2 posit compactness of $[0, T] \times X$.
The probability measures are $\inp{1}{\mu_0} = 1$ (by \eqref{eq:robust_meas_prob}), $\inp{1}{\mu_p} = 1$ (by \eqref{eq:robust_meas_liou} with $v(t,x)=1$).
The relaxed occupation measure $\mu$ is bounded with $\inp{1}{\mu} \leq T$ (by \eqref{eq:robust_meas_liou} under A1). 

Applying a test function $\zeta_j=1$ to the domination constraint \eqref{eq:robust_meas_dom} leads to
\begin{align}
    e \inp{1}{\mu} = A \inp{1}{\nu} + \inp{1}{\hat{\mu}} \Rightarrow e \inp{1}{\mu} \geq A \inp{1}{\nu}.
\end{align}
where the measure pairings are vectorized for convenience. The pairings $\inp{1}{\nu}$ are members of the $\inp{1}{\mu}$-scaled compact polytope $W$, proving that the constraint-slack measures $\hat{\mu}_j$ is bounded for each $j\in 1..m$.
\end{proof}

\begin{rmk}
    The signed measures $\nu_\ell$ in \eqref{eq:robust_meas} have unbounded \ac{TV} norm. Each  signed measure $\nu_\ell \in \mathcal{M}([0, T] \times X)$ can be decomposed into nonnegative measures by a Hahn-Jordan decomposition:
    \begin{align}
    \label{eq:hahn_jordan_robust}
        \nu_\ell &= \nu_\ell^+ - \nu_\ell^-, & & \nu_\ell^+, \ \nu_\ell^- \in \Mp{[0, T] \times X} & & \forall \ell=1..L \\
        & \nu_\ell^+ \perp \nu^-_\ell & & \forall \ell=1..L.     \label{eq:hahn_jordan_robust_perp}
    \end{align}

    A direct substitution of \eqref{eq:hahn_jordan_robust} into \eqref{eq:robust_meas} will leave the \ac{TV} norm $\norm{\nu_\ell}_{TV} = \inp{1}{\nu^+_\ell+\nu^-_\ell}$ as a possibly unbounded degree of freedom, because only the mass of the difference $\inp{1}{\nu^+_\ell-\nu^-_\ell}$ is constrained in \eqref{eq:robust_meas_dom}. Under the assumption that the \ac{SDR} set $W$ is compact, the measures $\nu^+_\ell$ and $\nu^-_\ell$ may be bounded by adding new mass constraints to \eqref{eq:robust_meas}:
    \begin{subequations}
        \label{eq:robust_new_mass}
    \begin{align}
        M_\ell^+ &= \max_{w \in W, \ w_\ell \geq 0} w_\ell, & 
        \inp{1}{\nu^+_\ell} &\leq \inp{1}{\mu} M_\ell^+ & & \forall \ell=1..L \\
        M_\ell^-&= \min_{w \in W, \ w_\ell \leq 0} w_\ell,& \inp{1}{\nu^-_\ell} &\leq \inp{1}{\mu}M^-_\ell  & & \forall \ell=1..L.
    \end{align}
    \end{subequations}
    The addition of \eqref{eq:robust_new_mass} will not change the optimum value $p^*$ of \eqref{eq:robust_meas}. However, the dual of \eqref{eq:robust_meas} with constraints in \eqref{eq:robust_new_mass} will be different from \eqref{eq:poly_peak_w_robust_sp}, and will no longer feature equality constraints in \eqref{eq:poly_peak_w_robust_sp_dual}.
\end{rmk}

\begin{rmk}
    The duality results of this appendix subsection may be extended to other compact \ac{SDR} sets $W$. Special caution must be taken notationally when referring to the adjoints of affine maps (from \eqref{eq:w_sdr}) and the dual spaces of cone-valued continuous functions $C([0, T] \times X \rightarrow K_s^*)'$. 
\end{rmk}

\begin{rmk}
    The work in \cite{majumdar2014convex} performs optimal control in the unit box $W = [-1, 1]^L$, resulting in measure programs that have the form of     \eqref{eq:hahn_jordan_robust} (excluding the orthogonality constraint \eqref{eq:hahn_jordan_robust_perp}) 
    with an additional bounded-mass constraint:
    \begin{align}
        \nu^+_\ell + \nu^-_\ell + \hat{\mu}_\ell &= \mu_\ell&  \nu_\ell^+ - \nu^-_\ell &= \nu_\ell & \forall \ell=1..L.
    \end{align}
    The work in \cite{korda2015controller} rescales the dynamics to ensure that  $W = [0, 1]^L$. The measure $\nu^-_\ell$ can be set to zero in the nonnegative box case, and the problem involves only nonnegative measures with $\nu_\ell = \nu_\ell^+$ for each $\ell$.
\end{rmk}

\subsection{Recovery}

Let $(v, \zeta)$ be a degree-$2k$ solution to the \ac{SOS} program \eqref{eq:robust_lie_polytope_sos} associated with \eqref{eq:poly_peak_w_robust_sp}.
Let $Q_0$ be the solved Gram matrix associated with the Lie constraint \eqref{eq:poly_peak_w_robust_sp_lie} (\ac{SOS} constraint \eqref{eq:robust_lie_polytope_sos_crit}), and let $\sigma_\ell$ be the vector of dual variables corresponding to the equality constraint \eqref{eq:poly_peak_w_robust_sp_dual} (finite-degree \eqref{eq:robust_lie_polytope_sos_eq}). By strong duality in the hierarchy (\cite{lasserre2009moments} and extensions from \cite[Theorem 4 and Lemma 4]{korda2014convex}), the \ac{SDP} dual variable to $Q_0$ is the moment matrix $\M_d[\bbm]$ in which $\bbm$ is a moment sequence of $\mu$. Similarly, the dual variables of $\sigma_\ell$ are moment sequences $\bbm_\ell$ of a signed measure $\mu_\ell$ for each $\ell=1..L$ (because every symmetric matrix may be expressed as the difference between two \ac{PSD} matrices).
An approximate control law $\hat{w}_\ell(t, x)$ for all $\ell=1..L$ may be recovered from the degree-$\leq 2k$ moments of $\bbm$ and the degree-$\leq k$ moments of $\bbm_\ell$ (written as $\bbm_\ell^{\leq d}$) by \cite[Equation 41]{majumdar2014convex}
\begin{subequations}
\label{eq:recover_law}
\begin{align}
    \mathbf{y}_\ell &= \M_d[\bbm]^{-1} (\bbm_\ell^{\leq d}) & 
    \hat{w}_\ell(t, x) &= \sum_{(\alpha, \beta)\in \N^{L+1}, \ \abs{\alpha} + \beta  \leq d} \mathbf{y}_{\ell \alpha \beta} x^\alpha t^\beta.
\end{align}
\end{subequations}

Controllers from \eqref{eq:recover_law} will converge in an $L_1$ sense to the optimal control law by \cite[Theorem 8]{majumdar2014convex} as the degree increases. It remains an 
open problem to quantify performance indices when deploying finite-degree recovered controls on system \eqref{eq:disturbance_affine}. 


%% file: appendix/app_problems.tex
\section{Analysis and Control Linear Programs}
\label{app:analysis_control_problem}

This appendix lists infinite-dimensional \acp{LP} in auxiliary functions for the distance estimation, reachable set estimation, and \ac{ROA} maximization problems. Assumptions A1-A4 are shared in all problems.
\subsection{Distance Estimation}
The Distance Estimation program under uncertainty $w(\cdot)$ is (up to a difference in signs from \cite{miller2021distance})
\begin{subequations}
\label{eq:poly_dist_w}
\begin{align}
    d^* = & \ \inf_{v(t,x), \gamma} \quad \gamma & \\
    & \gamma \geq v(0, x)  & &  \forall x \in X_0 \\
    & \Lie_f v(t, x, w) \leq 0 & & \forall (t, x, w) \in [0, T] \times X \times W \label{eq:poly_dist_w_lie}\\
    & v(t, x) \geq \phi(x) & & \forall (t, x) \in [0, T] \times X  \label{eq:poly_dist_w_cost} \\
    & \phi(x) \geq -c(x, y) & & \forall (x, y) \in  X \times Y.  \label{eq:poly_dist_w_closest} \\
    & \phi \in C(X), \ v \in C^1([0, T] \times X).
\end{align}
\end{subequations}

The distance of closest approach is $c^* = -d^*$.

\subsection{Reachable Set Estimation}






An infinite-dimensional \ac{LP} in continuous auxiliary functions $v(t, x)$ and $w(x)$ may be developed to outer-approximate the reachable set $X_T$  from \eqref{eq:reach_traj} \cite{henrion2013convex} as in
\begin{subequations}\label{eq:reach_cont}
\begin{align}
    d^* = & \inf_{v(t, x), \phi(x)} \textstyle\int_X \phi(x) dx & \\
    & v(0, x)  \leq 0 & &  \forall x \in X_0  \\ 
    & \phi(x) + v(T, x) \geq 1 & & \forall x \in X\\
    & \Lie_f v(t,x, w) \leq 0   & & \forall (t,x, w) \in [0, T] \times X \times W
 \label{eq:reach_cont_lie} \\
    & v(t,x) \in C^1([0, T] \times X) & \\
    & \phi(x) \in C_+(X).
\end{align}
\end{subequations}

At a degree-$d$ LMI relaxation, the set $\{x\in X \mid \phi(x) \geq 1\}$ is an outer approximation to the reachable set with volume bounds yielding the bounds $d_d^* \geq d_{d+1}^* \geq P^* = \textrm{vol}(X_T)$.
This sublevel set will converge in volume to the region of attractions (excluding sets of measure zero) as $d \rightarrow \infty$. 
The superlevel set approximations will be valid in containing $X_T$, except for possibly a set with Lebesgue measure zero (e.g. points, planes).
Inner approximations to the region of attraction can be performed through the methods in \cite{korda2013inner}.

\subsection{Region of Attraction Maximization}
The \ac{LP} in functions $v, \phi$ to perform \ac{ROA} maximization is \cite{henrion2013convex}
\begin{subequations}
\label{eq:roa_cont}
\begin{align}
    d^* = & \inf \textstyle\int_X \phi(x) dx & \\
    & v(T, x)  \geq 0 & &  \forall x \in X_T  \\ 
    & \phi(x) \geq 1 + v(0, x) & &\forall x \in X\\
    &\Lie_f v(t,x, w) \leq 0 & \label{eq:roa_cont_lie} & \forall (t,x, w) \in [0, T] \times X \times W\\
    & v(t,x) \in C^1([0, T] \times X) & \\
    & \phi(x) \in C_+(X). & 
\end{align}
\end{subequations}

The roles of $t=\{0, T\}$ and some signs are swapped in \eqref{eq:roa_cont} as compared to \eqref{eq:reach_cont}.

%% file: appendix/appendix_control_cost.tex
\section{Integral Costs and Robust Counterparts}
\label{app:integral_cost}

This appendix discusses Lie nonnegativity constraints with a cost $J(t, x, w)$:
\begin{align}
    \Lie_f v(t, x, w) + J(t, x, w) &\geq 0  & \forall (t, x, w) \in [0, T] \times X \times W. \label{eq:lie_control}
\end{align}

Constraints with \eqref{eq:lie_control} appear in \acp{OCP} involving integral costs:\begin{align}
    \inf_{w} \textstyle \int_0^T J(t, x(t), w(t)) dt & & \dot{x}(t) = f(t, x(t), w(t)), \ x(0) = x_0.
\end{align}
The expression in \eqref{eq:lie_control} is an inequality relaxation \cite{lewis1980relaxation} of the Hamilton-Jacobi-Bellman constraint \cite{liberzon2011calculus}:
\begin{align}
    \min_{w \in W} \Lie_f v(t, x, w) + J(t, x, w) = 0 & & \forall (t, x) \in [0, T] \times X. 
\end{align}

Integral costs with peak estimation are discussed in Equation (5.2) of \cite{fantuzzi2020bounding}. 


This appendix will assume that Assumptions A1-A4 are active. 

Representations of \eqref{eq:lie_control} in standard robust form \eqref{eq:lin_robust} will be worked out for the specific cases of 
$L_\infty$, $L_1$, and quadratic running costs. Results will be reported as the combination of an extended \ac{SDR} uncertainty set $\tilde{W}$ (such that $\pi^w \tilde{W} = w$) and terms $(a, b)$ to form \eqref{eq:lin_robust}.



\subsection{L-infinity Running Cost}

This subsection will involve a running cost $J(t, x, w) = \norm{C w}_\infty$ for a matrix $C \in \R^{c \times L}$ with $c \geq L$ and $\rank{C}=L$. A new term $\tilde{w} \in \R$ may be  introduced to form the lifted \ac{SDR} uncertainty set
\begin{align}
    \tilde{W}_{\infty} = \{w \in W, \ \tilde{w} \in \R: \1_c \tilde{w} - Cw \geq 0, \ \1_c \tilde{w} + C w \geq 0\}. \label{eq:w_infinity_set}
\end{align}
The weighted $L_\infty$-running cost Lie term from \eqref{eq:lie_control} is
\begin{align}
    & \Lie_f v(t, x, w) + \norm{C w}_\infty \geq 0 & & \forall (t, x, w) \in [0, T] \times X \times W. \nonumber \\
     =& \Lie_f v(t, x, w) + \tilde{w} \geq 0 & & \forall (t, x, w, \tilde{w}) \in [0, T] \times X \times \tilde{W}_{\infty}.
     \label{eq:lie_control_infty}
\end{align}
The correspondence in \eqref{eq:lin_robust} for the $L_\infty$ running cost is  \eqref{eq:lie_control_infty} $\tilde{w}$ by
\begin{subequations}
\begin{align}
    b_0 &= \Lie_{f_0} v(t, x, w) & a_0 &= 0 \\
    b_\ell &= f_\ell \cdot \nabla_x v(t, x, w) & a_\ell &= 0 & & \forall \ell\in1..L \\
    b_{L+1} &= 1 & a_{L+1} &= 0.
\end{align}
\end{subequations}

\subsection{L1 Running Cost}

This subsection has a running cost of $J(t, x, w) = \norm{w}_1$, as performed by \cite{majumdar2013control}. The standard $L_1$ lift as reported in \cite{gouveia2013lifts} introduces $\tilde{w} \in \R^L$ under the constraint 
\begin{align}
    \tilde{W}_{1} = \{w \in W, \ \tilde{w} \in \R^L: \tilde{w} - w \geq 0, \ \tilde{w} + w \geq 0\}.
\end{align}

The $L_1$-running cost Lie term from \eqref{eq:lie_control} is
\begin{align}
    & \Lie_f v(t, x, w) + \norm{w}_1 \geq 0 & & \forall (t, x, w) \in [0, T] \times X \times W. \nonumber \\
     =& \Lie_f v(t, x, w) + \1_L^T \tilde{w} \geq 0 & & \forall (t, x, w, \tilde{w}) \in [0, T] \times X \times \tilde{W}_{1}.
     \label{eq:lie_control_1}
\end{align}
The correspondence in \eqref{eq:lin_robust} for the $L_1$ case \eqref{eq:lie_control_1} is 
\begin{subequations}
\begin{align}
    b_0 &= \Lie_{f_0} v(t, x, w) & a_0 &= 0 \\
    b_\ell &= f_\ell \cdot \nabla_x v(t, x, w) & a_\ell &= 0 & & \forall \ell\in1..L \\
    b_{\ell'} &= 1 & a_{\ell'} &= 0 & & \forall \ell' \in L+1..2L.
\end{align}
\end{subequations}

\subsection{Quadratic Running Cost}

This subsection will discuss the standard convex quadratic cost:
\begin{align}
    J(t,x,w) &= x^T P x + w^T R w + 2 w^T N x& P \in \psd_{+}^n, \ R \in \psd^L_{+}, \ N \in \R^{L \times n}.
\end{align}

 Let $\Xi = [Q, N^T; \; N, R]$ be a matrix with factorization $(\Xi^{1/2})^T \Xi^{1/2} = \Xi$. 
The cone description with mixed quadratic uncertainty is \cite{alizadeh2003second}
\begin{align}
    \tilde{W}_2 = \{w \in W, \tilde{w} \in \R: (\Xi^{1/2} [x; w], \tilde{w}, 1/2) \in Q^L_{r} \}. \label{eq:cone_mix_w2}
\end{align}

The quadratic-cost Lie expression from \eqref{eq:lie_control} is
\begin{subequations}
\begin{align}
    & \Lie_f v(t, x, w) + x^T P x + w^T R w + 2 w^T N x \geq 0 & & \forall (t, x, w) \in [0, T] \times X \times W \nonumber \\
     =& \Lie_f v(t, x, w) + \tilde{w} \geq 0 & & \forall (t, x, w, \tilde{w}) \in [0, T] \times X \times \tilde{W}_{2}.
     \label{eq:lie_control_quad}
\end{align}
\end{subequations}
The correspondence in \eqref{eq:lin_robust} for the mixed quadratic case case \eqref{eq:lie_control_quad} is
\begin{subequations}
\begin{align}
    b_0 &= \Lie_{f_0} v(t, x, w) & a_0 &= 0 \\
    b_\ell &= f_\ell \cdot \nabla_x v(t, x, w) & a_\ell &= 0 & & \forall \ell\in1..L \\
    b_{L+1} &= 1 & a_{L+1} &= 0.
\end{align}
\end{subequations}

\begin{lem}
Continuity and approximability (Theorem \ref{thm:multipliers_cont_poly}) is preserved in the mixed set \eqref{eq:cone_mix_w2} (after accounting for the sign changes with the strict $\Lie_f v > 0$). 
\end{lem}
\begin{proof}
Let $[\Xi^{1/2}_x, \; \Xi_w^{1/2}]$ be a column-wise partition of $\Xi^{1/2}$ corresponding to the $x$ and $w$ multiplications. The rotated  \ac{SOC} constraint in \eqref{eq:cone_mix_w2} may be expressed with parameters
\begin{subequations}
\begin{align}
    A_{\text{mix}} &= \begin{bmatrix} \Xi^{1/2}_w & 0 \\ \0_{L \times 1} & 1 \\ \0_{L \times 1} & 0\end{bmatrix} & e_{\text{mix}} &= \begin{bmatrix}  \Xi^{1/2}_x x \\0 \\1/2
    \end{bmatrix} \\
    G_{\text{mix}} &= \varnothing & K_{\text{mix}} &= Q^L_r   
\end{align}
\end{subequations}
forming the conic constraint
\begin{align}
    A_{\text{mix}}[w; \tilde{w}] + e_{\text{mix}} \in K_{\text{mix}}.
\end{align}
Now consider the assumptions in Section \ref{sec:robust_continuity}. A3' is satisfied because $e_{\text{mix}}$ is a continuous (affine) function of $x$ and does not involve $t$. A4' is also satisfied because $A_{\text{mix}}$ is constant in $(t, x)$ and $ G_{\text{mix}} = \varnothing$. Assumptions A1-A4 ensure that A1' and A2' are fulfilled, completing the proof.
\end{proof}